\documentclass[reqno,letterpaper]{gthack}   
\usepackage[margin=1.15in, left=1.5in, right=1.3in,bottom=1.05in,top=1.15in]{geometry}

\usepackage[numbers]{natbib}
\setcitestyle{open={[},close={]}}

\usepackage{xfrac}

\titleformat{\section}
  {\normalfont\fontsize{14}{5}\bfseries}{\thesection}{1em}{}

\usepackage{wrapfig}

\usepackage{graphicx}
\usepackage{enumitem}
\usepackage{tikz-cd}
\allowdisplaybreaks

\usepackage{comment}
\usepackage[margin=.45cm,font=small,labelfont=small]{caption}

\title{Corks, covers, and complex curves}
\author{Kyle Hayden} \address{Columbia University, New York, NY 10027} 
\email{hayden@math.columbia.edu}

\newcommand{\hl}[1]{#1}

\theoremstyle{plain}
\newtheorem{thm}{Theorem}[section]   \newtheorem{lem}[thm]{Lemma}
              \newtheorem{prop}[thm]{Proposition}

\newtheorem*{ques*}{Question}

\newtheorem*{thmcusp*}{Theorem B}
\newtheorem*{thm-plain}{Theorem}

\newtheorem{mainthm}{Theorem}[section]

\theoremstyle{definition}
\newtheorem{defn}[thm]{Definition}    

\newtheorem{rem}[thm]{Remark}       \newtheorem*{rem*}{Remark}             
\theoremstyle{remark}
\newtheorem{ex-main}[thm]{Example}

\newcommand{\cc}{\mathbb{C}}
\newcommand{\rr}{\mathbb{R}}
\newcommand{\zz}{\mathbb{Z}}
\newcommand{\dd}{\mathbb{D}}

\renewcommand{\S}{\textsection}

\newcommand{\kerncomma}{\kern-0.14em, \kern0.05em}

\begin{document}

\begin{abstract}
\vspace{-.125in}

We show that $\cc^2$ contains pairs of properly embedded, smooth complex curves that are isotopic through homeomorphisms but not diffeomorphisms of $\cc^2$. The construction is based on realizing corks as branched covers of holomorphic disks in the 4-ball. These disks can also be described using exotic factorizations of quasipositive braids.
\end{abstract}

\maketitle

\titleformat{\section}{\large\bfseries}{}{0pt}{\center \thesection.  }

\titlespacing{\section}{0pt}{*4}{*1.5}

\titleformat{\subsection}[runin]{\bfseries}{}{0pt}{\thesubsection \ \  }
\titlespacing{\subsection}{0pt}{*1}{*1.65}

\vspace{-.65in}

\section{Introduction}\label{sec:intro}
  
A pair of smooth surfaces in a 4-manifold are \emph{exotically knotted} if they are isotopic through  ambient homeomorphisms but not ambient diffeomorphisms.  
While many 4-manifolds contain a wealth of exotically knotted surfaces, examples in the simplest 4-manifolds  are more elusive. 
In particular, it remains a major open problem to determine   whether  $\rr^4$ contains exotically knotted pairs  of closed, orientable surfaces. In contrast, $\rr^4$ is known to contain exotic pairs of closed, nonorientable surfaces  \cite{fkv}. 

As with many difficult problems in 4-manifold topology, additional tools are available when we shift to the noncompact and relative settings. 
 In unpublished work from 1985, Freedman constructed a smooth, properly embedded copy of $\rr^2$ in $\rr^4$ that is topologically but not smoothly isotopic to the standard plane $\rr^2 \! \times \! 0$.  Gompf expanded on this example in \cite{gompf:menagerie} and provides an extensive investigation of exotic planes in $\rr^4$ in forthcoming work \cite{gompf:planes}. In the relative setting, there has been recent work exhibiting exotically knotted pairs of properly embedded, orientable surfaces with boundary in the 4-ball; see \cite{jmz:exotic,hayden:disks,hkkmps}. 

This paper sheds light on two
seemingly disparate aspects of this knotting problem: \linebreak the existence of exotically knotted pairs of  complex curves in $\cc^2$, and the development of techniques in the relative and noncompact settings that are better suited for adaptation to the closed setting.

\subsection{Exotically knotted complex curves in $\boldsymbol{\mathbf{C}^2}$.}  
Much remains unknown about the topology of  complex plane curves, especially non-algebraic curves. Nevertheless, the complex setting is generally known to impose significant rigidity. For example, in a closed, simply connected K\"ahler surface,  any two smooth, closed complex curves that are  homologous (which includes those that are topologically isotopic)  are smoothly isotopic through complex curves \cite[\S1]{fs:symplectic}.

Even in the noncompact setting of $\cc^2$,  an \emph{algebraic} curve  $V$ (i.e., the zero-locus of a polynomial) is largely governed by the curve's  ``link at infinity''\kerncomma an iterated torus link defined by $L_\infty=V \cap S^3_R$ for $R \gg 1$.  
Neumann showed that if $V$ is a generic algebraic curve or if $L_\infty$ is connected (i.e., a knot), then $V$ is isotopic to a standard Seifert surface for $L_\infty$ that has been extended cylindrically  out to infinity  in $\cc^2$ \cite{neumann}. This, for example, implies that there can be no exotically knotted algebraic curves in $\cc^2$ with the topological type of a once-punctured surface. (Moreover, we expect that there are no exotic pairs of algebraic curves in $\cc^2$ whatsoever.)

Despite the rigidity exhibited in the algebraic and closed settings, we show that the smooth and topological categories do \emph{not} coincide for general  complex curves in $\cc^2$.

\begin{mainthm}\label{thm:complex}
There are infinitely many pairs of proper holomorphic curves in $\cc^2$ that are isotopic through ambient
 homeomorphisms but not ambient diffeomorphisms. \end{mainthm}

These examples realize all possible topological types (with finite topology) except the plane. (Note that any proper holomorphic curve in $\cc^2$ must be an orientable, open surface with no closed components.)   
Theorem~\ref{thm:complex} gives a positive answer to \cite[Question~1.4]{hayden:disks}, which asks whether $\cc^2$ contains exotically knotted pairs of complex curves. It also offers evidence towards a positive answer to \cite[Question~6.13]{gompf:planes}, which concerns the special case of exotic complex lines. 

\subsection{From corks to complex curves.} Our underlying construction is based on a connection between corks and knotted disks in the 4-ball.  A  \emph{cork} is a compact, contractible 4-manifold whose boundary is equipped with a diffeomorphism (its \emph{boundary twist}) that fails to extend to a diffeomorphism of the 4-manifold's interior. These objects lie at the heart of exotic phenomena in 4-manifold topology. For example, any pair of smooth, closed, simply connected 4-manifolds that are homeomorphic are related by removing a cork and regluing it using its boundary twist \cite{CFHS,matveyev}. 

We show that many corks arise as double branched covers of slice disks in the 4-ball. Surprisingly, in many cases, these slice disks can be realized as compact pieces of complex curves in $B^4 \subset \cc^2$. Examples include the famous Mazur cork \cite{akbulut:cork} and positron cork \cite{akbulut-matveyev:decomp}, depicted in Figures~\ref{fig:positron}-\ref{fig:mazur} alongside their associated slice disks.

\begin{mainthm}\label{thm:cork}
There are infinitely many corks $W$ that each arise as the double branched cover of a holomorphically embedded slice disk $D$ in $B^4 \subset \cc^2$. Moreover, the boundary twist on $\partial W$ is the lift of an involution of the slice knot $(S^3,K)$ with $K=\partial D$ that fails to extend to a diffeomorphism of the pair $(B^4,D)$.
\end{mainthm}

\begin{figure}[b]\center
\def\svgwidth{\linewidth}
\begingroup%
  \makeatletter%
  \providecommand\color[2][]{%
    \errmessage{(Inkscape) Color is used for the text in Inkscape, but the package 'color.sty' is not loaded}%
    \renewcommand\color[2][]{}%
  }%
  \providecommand\transparent[1]{%
    \errmessage{(Inkscape) Transparency is used (non-zero) for the text in Inkscape, but the package 'transparent.sty' is not loaded}%
    \renewcommand\transparent[1]{}%
  }%
  \providecommand\rotatebox[2]{#2}%
  \newcommand*\fsize{\dimexpr\f@size pt\relax}%
  \newcommand*\lineheight[1]{\fontsize{\fsize}{#1\fsize}\selectfont}%
  \ifx\svgwidth\undefined%
    \setlength{\unitlength}{1314.06000092bp}%
    \ifx\svgscale\undefined%
      \relax%
    \else%
      \setlength{\unitlength}{\unitlength * \real{\svgscale}}%
    \fi%
  \else%
    \setlength{\unitlength}{\svgwidth}%
  \fi%
  \global\let\svgwidth\undefined%
  \global\let\svgscale\undefined%
  \makeatother%
  \begin{picture}(1,0.25939198)%
    \lineheight{1}%
    \setlength\tabcolsep{0pt}%
    \put(0.19811594,0.19890722){\color[rgb]{0.30196078,0.30196078,0.30196078}\makebox(0,0)[lt]{\lineheight{1.25}\smash{\begin{tabular}[t]{l}$0$\end{tabular}}}}%
    \put(0,0){\includegraphics[width=\unitlength,page=1]{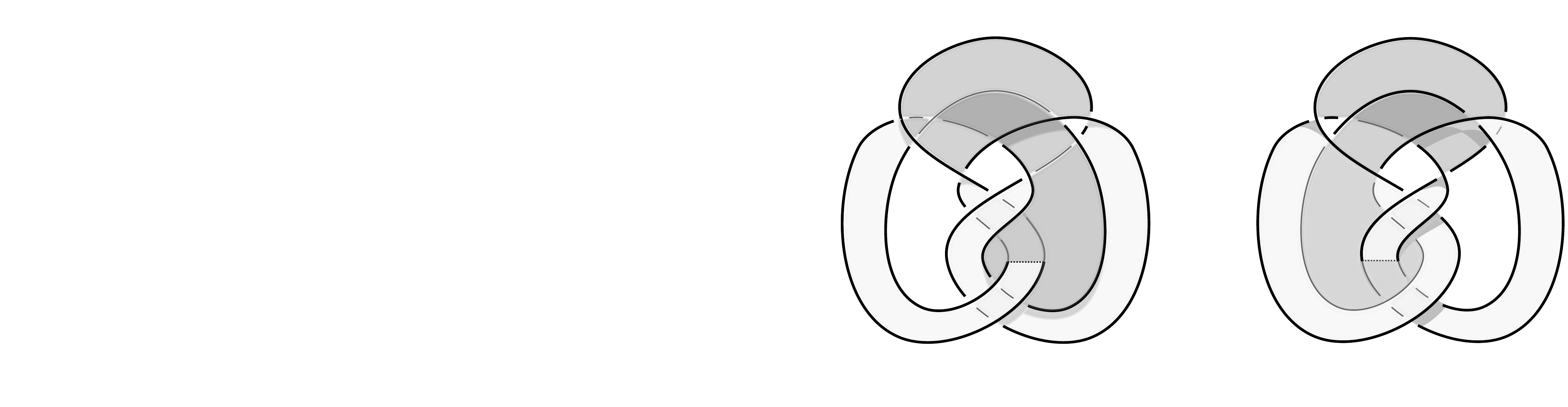}}%
    \put(0.42035656,0.04539086){\color[rgb]{0.4,0.4,0.4}\makebox(0,0)[lt]{\lineheight{1.25}\smash{\begin{tabular}[t]{l}$0$\end{tabular}}}}%
    \put(0,0){\includegraphics[width=\unitlength,page=2]{positron-intro.pdf}}%
    \put(0.24626862,0.1260671){\makebox(0,0)[lt]{\lineheight{1.25}\smash{\begin{tabular}[t]{l}$=$\end{tabular}}}}%
    \put(0.6221175,0.00226906){\makebox(0,0)[lt]{\lineheight{1.25}\smash{\begin{tabular}[t]{l}$D$\end{tabular}}}}%
    \put(0.88466228,0.00226906){\makebox(0,0)[lt]{\lineheight{1.25}\smash{\begin{tabular}[t]{l}$D'$\end{tabular}}}}%
  \end{picture}%
\endgroup%

\caption{The positron cork on the left can be realized as the double branched cover of $B^4$ over each of the disks $D$ and $D'$ on the right.}\label{fig:positron}
\end{figure}

\begin{figure}\center
\def\svgwidth{\linewidth}
\begingroup%
  \makeatletter%
  \providecommand\color[2][]{%
    \errmessage{(Inkscape) Color is used for the text in Inkscape, but the package 'color.sty' is not loaded}%
    \renewcommand\color[2][]{}%
  }%
  \providecommand\transparent[1]{%
    \errmessage{(Inkscape) Transparency is used (non-zero) for the text in Inkscape, but the package 'transparent.sty' is not loaded}%
    \renewcommand\transparent[1]{}%
  }%
  \providecommand\rotatebox[2]{#2}%
  \newcommand*\fsize{\dimexpr\f@size pt\relax}%
  \newcommand*\lineheight[1]{\fontsize{\fsize}{#1\fsize}\selectfont}%
  \ifx\svgwidth\undefined%
    \setlength{\unitlength}{1319.78059063bp}%
    \ifx\svgscale\undefined%
      \relax%
    \else%
      \setlength{\unitlength}{\unitlength * \real{\svgscale}}%
    \fi%
  \else%
    \setlength{\unitlength}{\svgwidth}%
  \fi%
  \global\let\svgwidth\undefined%
  \global\let\svgscale\undefined%
  \makeatother%
  \begin{picture}(1,0.24282994)%
    \lineheight{1}%
    \setlength\tabcolsep{0pt}%
    \put(0.18032287,0.19639491){\color[rgb]{0.30196078,0.30196078,0.30196078}\makebox(0,0)[lt]{\lineheight{1.25}\smash{\begin{tabular}[t]{l}$0$\end{tabular}}}}%
    \put(0,0){\includegraphics[width=\unitlength,page=1]{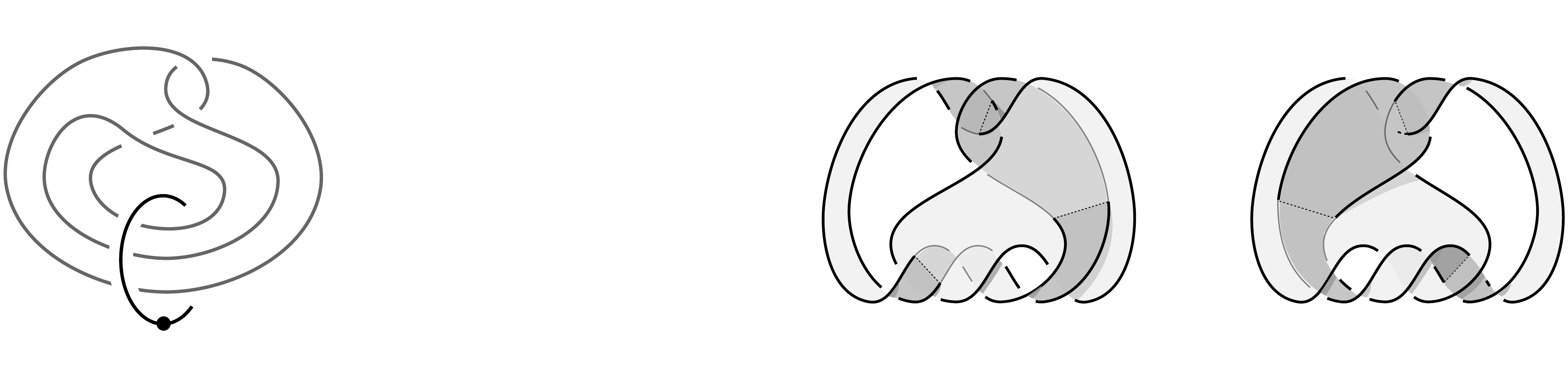}}%
    \put(0.4112859,0.06046396){\color[rgb]{0.30196078,0.30196078,0.30196078}\makebox(0,0)[lt]{\lineheight{1.25}\smash{\begin{tabular}[t]{l}$0$\end{tabular}}}}%
    \put(0,0){\includegraphics[width=\unitlength,page=2]{mazur-intro.pdf}}%
    \put(0.23688993,0.11456186){\makebox(0,0)[lt]{\lineheight{1.25}\smash{\begin{tabular}[t]{l}$=$\end{tabular}}}}%
  \end{picture}%
\endgroup%

\vspace{-.075in}

\caption{The Mazur cork is the double branched cover of $B^4$ over each disk on the right.}\label{fig:mazur}
\end{figure}

To connect this to the study of knotted surfaces,  observe that this construction naturally gives rise to \emph{pairs} of disks: by extending the involution of $(S^3,K)$ to $B^4$ and considering the image of $D$, we obtain a second disk $D'$ bounded by $K$. The fact that the cork's boundary twist cannot extend smoothly over the cork implies that the resulting pairs $(B^4,D)$ and $(B^4,D')$ are not diffeomorphic rel boundary; see \S\ref{sec:invertible}.  However, we show that there are cases in which $D$ and $D'$ are isotopic through homeomorphisms of $B^4$ rel boundary. These form the building blocks for our larger exotic surfaces. 

Though important in their own right, the noncompact and relative settings may also serve as inroads to the subtler closed setting. With this in mind, we aim to develop techniques that could have traction with closed surfaces. Existing techniques fall short on two measures:  On one hand, additional obstructive tools are often available in the relative setting of 4-manifolds with boundary. For example, the techniques previously used to distinguish exotic, properly embedded surfaces in $B^4$ do not obstruct the surfaces' interiors from becoming isotopic in the open 4-ball. On the other hand, the exotic phenomena detected in noncompact settings is often constructed via infinite processes that are difficult to adapt to the compact setting. For example,  previously known exotic orientable surfaces in $\rr^4$ (e.g., \cite{gompf:menagerie,gompf:planes}) require infinitely many critical points and cannot arise as the interior of a compact surface in $B^4$. In contrast, by carefully embedding the disks from Theorem~\ref{thm:cork} into larger surfaces, we prove:

\begin{mainthm}\label{thm:ball}
Any compact, connected, orientable surface with boundary, other than the disk, admits pairs of smooth, proper embeddings into the closed 4-ball that are exotically knotted and whose interiors remain exotically knotted in the open 4-ball.
\end{mainthm}

To distinguish these surfaces, we compare their associated double branched covers of $B^4$,   and we use an adjunction inequality \cite{lisca-matic} to distinguish the \emph{interiors} of these 4-manifolds;  this in turn distinguishes the surfaces' interiors in the open 4-ball. (This cannot distinguish  slice disks, whose double branched covers are $\mathbb{Q}$-homology balls.)

After establishing Theorems~\ref{thm:cork} and \ref{thm:ball}, there are two key steps remaining in the proof of Theorem~\ref{thm:complex}: We first arrange the larger surfaces in $B^4$ from Theorem~\ref{thm:ball} to arise as compact pieces of complex curves, and we then carefully re-embed these surfaces' interiors as proper complex curves in $\cc^2$. We note that this last step is nontrivial because the open 4-ball is not biholomorphically equivalent to $\cc^2$. Instead, we achieve this re-embedding by situating our initial complex curves inside of \emph{Fatou-Bieberbach domains} (cf \cite{bkw}); see  \S\ref{sec:plane} for details. This re-embedding may change the surfaces' isotopy types, but the more robust obstruction used to establish Theorem~\ref{thm:ball} persists, distinguishing these curves' images in $\cc^2$.

To construct compact pieces of complex curves, we use Rudolph's technology of quasipositive braids and braided surfaces \cite{rudolph:qp-alg,rudolph:braided-surface}.  
Recall that an element of the braid group is \emph{quasipositive} if it is a  product of conjugates of the positive Artin generators, i.e., it can be written as
$$\prod \omega_i \sigma_{j_i} \omega_i^{-1},$$
where each $\omega_i$ is a word in the braid group and each $\sigma_{j_i}$ is a positive Artin generator.  As illustrated in Figure~\ref{fig:braid-and-surface}, each quasipositive braidword determines a \emph{positively braided surface}; see \cite[\S2]{rudolph:braided-surface} and \S\ref{sec:braid} below.   Rudolph showed that such a surface can be realized  as a compact piece of an algebraic curve. Crucially, this surface depends not only on the braid but on the quasipositive factorization itself, allowing for subtle control of the resulting surface (c.f., \cite{bkw,baykur-vhm,hayden:cross,oba:surfaces}). We show that the disks from Figure~\ref{fig:positron} arise from  ``exotic'' braid factorizations.

\begin{figure}\center
\includegraphics[width=\linewidth]{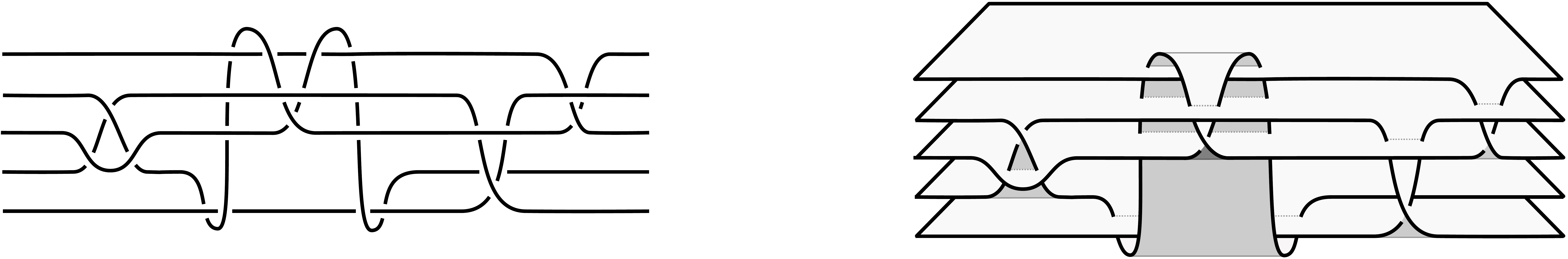}
\caption{The quasipositive braid $\beta$ and the associated positively braided surface.}
\label{fig:braid-and-surface}
\end{figure}

\begin{mainthm}\label{thm:factor}
Let $D$ and $D'$ denote the holomorphic disks in $B^4 \subset \cc^2$ 
obtained as braided surfaces from the following quasipositive words in the 5-stranded braid group:
\begin{align*}
\beta&=(\sigma_{2}\sigma_{3}\sigma_{2}^{-1})(\sigma_1^{-2} \sigma_2 \sigma_3 \sigma_4^2 \sigma_3^{-1} \sigma_2 \sigma_3 \sigma_4^{-2}  \sigma_3^{-1} \sigma_2^{-1} \sigma_1^2)(\sigma_3^{-1} \sigma_2 \sigma_1 \sigma_2^{-1} \sigma_3)(\sigma_4^{-1} \sigma_3 \sigma_4)\\
\beta'&=\sigma_{2}(w \sigma_{2}^{-1}\sigma_{1}\sigma_{2}w^{-1})(w\sigma_{2}^{-1}\sigma_{3}\sigma_{1}\sigma_{2}\sigma_{1}^{-1}\sigma_{3}^{-1}\sigma_{2}w^{-1}) (w\sigma_{3}^2\sigma_{4}\sigma_{3}^{-2}w^{-1}),
\end{align*}
where $w = \sigma_3 \sigma_4^{-1} \sigma_1^{-1} \sigma_3^{-2} \sigma_2^{-1}   \sigma_1^{-1} \sigma_3^{-1}$. There is a braid isotopy from $\beta$ to $\beta'$  extending to a topological isotopy from $D$ to $D'$ but not to a smooth isotopy from $D$ to $D'$.
\end{mainthm}

The construction of inequivalent braided surfaces through braid factorizations forms a close analog of the construction of inequivalent Lefschetz fibrations through monodromy factorizations. In particular, taking branched covers, the braid factorizations from Theorem~\ref{thm:factor} lift to products of positive Dehn twists on a genus two surface with one boundary component. This produces a pair of Lefschetz fibrations on the positron cork such that replacing one factorization with the other corresponds to performing the cork twist. For previous work relating monodromy factorizations and inequivalent Lefschetz fibrations, see \cite{endo-mark-vhm,baykur-korkmaz:exotic,akhmedov-katzarkov}; also see \cite{ukida} for Lefschetz fibrations on the Mazur cork (which do not arise as lifts of braids). 

The literature already contains a variety of inequivalent monodromy factorizations giving rise to Lefschetz fibrations whose total spaces are homeomorphic but not diffeomorphic. In contrast, previous  examples of braided surfaces arising from inequivalent braid factorizations (e.g., \cite{rudolph:braided-surface,auroux:factorizations,baykur-vhm,oba:surfaces}) are not topologically equivalent, as they are distinguished by the fundamental groups of their complements or the homotopy types of their branched covers. The braid factorizations in  Theorem~\ref{thm:factor} give the first examples in which the corresponding braided surfaces are exotically knotted.

\emph{Organization.} In \S\ref{sec:invertible}, we discuss the underlying topological construction, with a focus on the disks $D$ and $D'$ from Figure~\ref{fig:positron}, and then prove Theorem~\ref{thm:cork}. In \S\ref{sec:ball}, we introduce the rest of our key building blocks in the 4-ball and prove Theorem~\ref{thm:ball}. In \S\ref{sec:braid}, we discuss braided surfaces and recast the building blocks from \S\ref{sec:invertible}-\ref{sec:ball} as compact pieces of algebraic curves in $\cc^2$; the proof of Theorem~\ref{thm:factor} is given in this section.  Finally, in \S\ref{sec:plane}, we discuss Fatou-Bieberbach domains and complete the proof of Theorem~\ref{thm:complex}.

\smallskip

\emph{Conventions.} Unless specified otherwise, manifolds will be assumed to be smooth and orientable.  Any connected, orientable, properly embedded surface $F \subset B^4$ has $H_1(B^4 \setminus F) \cong \zz$, so we may unambiguously take the double branched cover of $B^4$ over $F$ to mean the branched cover associated to the homomorphism $\pi_1(B^4 \setminus F) \to H_1(B^4 \setminus F) \cong \zz \to \zz/2$ that maps each meridian of $F$ to $1 \in \zz/2$.

\smallskip

\emph{Acknowledgements.}  Thanks to Marco Golla and Bob Gompf for stimulating conversations on these topics, and to Siddhi Krishna and Lisa Piccirillo for helpful comments. This work was supported by NSF grants DMS-1803584 and DMS-2114837.

\vspace{-.25in}

\section{From corks to knotted disks}\label{sec:invertible}

\subsection{The topological construction.} \label{subsec:invertible} 
We begin by recalling that a link $L \subset S^3$ is \emph{strongly invertible} if there is an orientation-preserving involution of $S^3$ that preserves $L$ and induces on every component of $L$ an involution with exactly two fixed points.  Montesinos  \cite{montesinos} proved that if $L$ is  strongly invertible, then any 3-manifold obtained by Dehn surgery along $L$ is a double branched cover of $S^3$ along a link.

A key example is illustrated in Figure~\ref{fig:invertible}, which exhibits the boundary of the positron cork (from Figure~\ref{fig:positron}) as the double branched cover of a knot $K$ in $S^3$. The figure depicts two involutions of the underlying link $L$: a strong involution $\rho$ (which is used to produce $K$) and an involution $\tau$ that exchanges the link components (and which descends to a symmetry of $K$).

\begin{figure}[b]\center
\vspace{-.125in}
\def\svgwidth{.97\linewidth}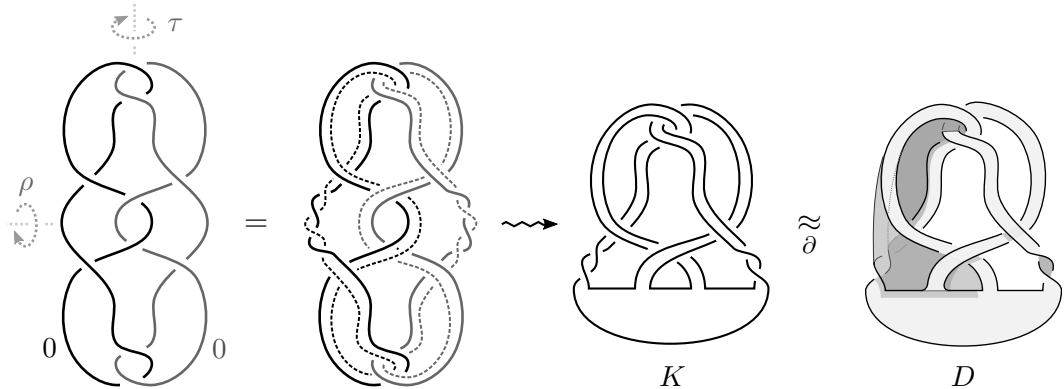
\caption{The boundary of the positron cork is given by surgery on the invertible link on the left, and is realized as the double branched cover of the slice knot $K=\partial D$.}\label{fig:invertible}
\end{figure}

In fact, a 4-dimensional extension holds: The positron cork itself is the double branched cover of $B^4$ along the slice disk $D$ shown in Figure~\ref{fig:invertible}. (This disk $D$ is isotopic  to the disk in Figure~\ref{fig:positron}; use the isotopy in Figure~\ref{fig:iso-K}, followed by a 180$^\circ$ rotation along the $x$-axis.) Moreover, while the involution $\tau$ descends to an involution of $K$, it does \emph{not} extend to an involution of the slice disk $D$.  
 Similar examples can be obtained from other simple corks, including the Mazur cork; see Figure~\ref{fig:mazur}. 

By considering the image of the disk $D$ under a natural extension of the involution of $K$, we obtain a second disk $D'$ bounded by $K$. Because the involution of $K$ induced by $\tau$ does not extend to a diffeomorphism of the pair $(B^4,D)$, the pairs $(B^4,D)$ and $(B^4,D')$ are not diffeomorphic rel boundary. However, in this example, we will see that these disks are \emph{topologically} isotopic rel boundary.

\begin{figure}\center
\def\svgwidth{.9\linewidth}
\begingroup%
  \makeatletter%
  \providecommand\color[2][]{%
    \errmessage{(Inkscape) Color is used for the text in Inkscape, but the package 'color.sty' is not loaded}%
    \renewcommand\color[2][]{}%
  }%
  \providecommand\transparent[1]{%
    \errmessage{(Inkscape) Transparency is used (non-zero) for the text in Inkscape, but the package 'transparent.sty' is not loaded}%
    \renewcommand\transparent[1]{}%
  }%
  \providecommand\rotatebox[2]{#2}%
  \newcommand*\fsize{\dimexpr\f@size pt\relax}%
  \newcommand*\lineheight[1]{\fontsize{\fsize}{#1\fsize}\selectfont}%
  \ifx\svgwidth\undefined%
    \setlength{\unitlength}{1406.79749977bp}%
    \ifx\svgscale\undefined%
      \relax%
    \else%
      \setlength{\unitlength}{\unitlength * \real{\svgscale}}%
    \fi%
  \else%
    \setlength{\unitlength}{\svgwidth}%
  \fi%
  \global\let\svgwidth\undefined%
  \global\let\svgscale\undefined%
  \makeatother%
  \begin{picture}(1,0.233411)%
    \lineheight{1}%
    \setlength\tabcolsep{0pt}%
    \put(0,0){\includegraphics[width=\unitlength,page=1]{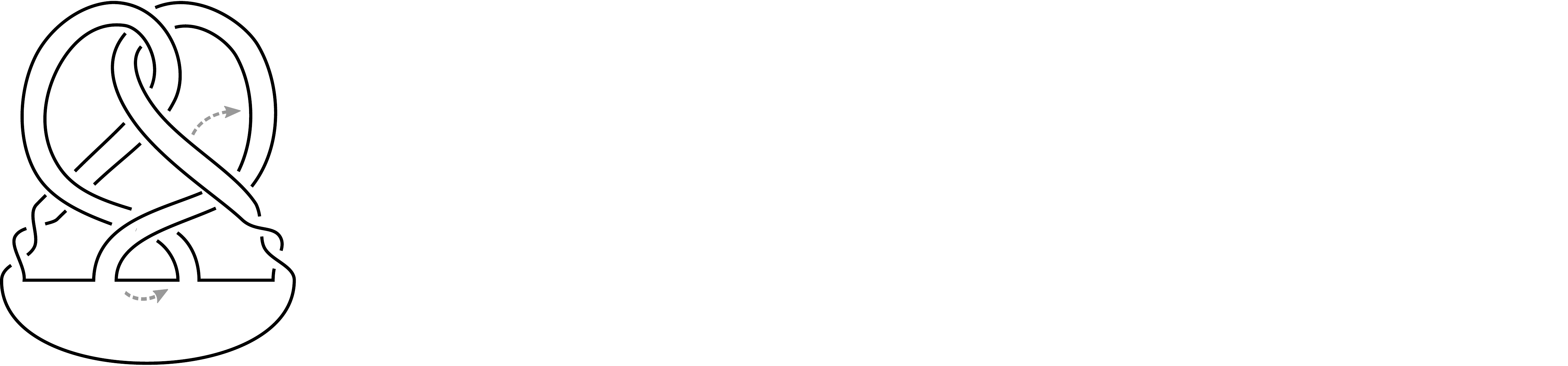}}%
    \put(0.20587121,0.11037981){\makebox(0,0)[lt]{\lineheight{1.25}\smash{\begin{tabular}[t]{l}$=$\end{tabular}}}}%
    \put(0.47371292,0.11037981){\makebox(0,0)[lt]{\lineheight{1.25}\smash{\begin{tabular}[t]{l}$=$\end{tabular}}}}%
    \put(0.73265824,0.11037981){\makebox(0,0)[lt]{\lineheight{1.25}\smash{\begin{tabular}[t]{l}$=$\end{tabular}}}}%
    \put(0,0){\includegraphics[width=\unitlength,page=2]{iso-disks.pdf}}%
  \end{picture}%
\endgroup%

\caption{Relating two diagrams for $K$.}\label{fig:iso-K}
\end{figure}

\subsection{Topological isotopy.} A pair of surfaces in $B^4$ are \emph{topologically} isotopic rel boundary if they are  isotopic through homeomorphisms of $B^4$ that fix $S^3$.

\begin{prop}\label{prop:top-iso}
The disks $D$ and $D'$ are topologically isotopic rel boundary.
\end{prop}

The proof is based on the following result of Conway and Powell. 

\begin{thm}[\cite{conway-powell}]\label{thm:conway-powell}
Any smooth, properly embedded disks in $B^4$ with the same boundary and whose complements have $\pi_1 \cong \zz$  are topologically isotopic rel boundary.
\end{thm}

The results in \cite{conway-powell} are stated for \emph{homotopy ribbon} disks, i.e., $D \subset B^4$ such that the inclusion $S^3 \setminus \partial D \hookrightarrow B^4 \setminus D$  induces a surjection on fundamental groups. When $\pi_1(B^4 \setminus D)$ is infinite cyclic, this inclusion-induced map is always surjective.

\begin{figure}[b]\center
\def\svgwidth{\linewidth}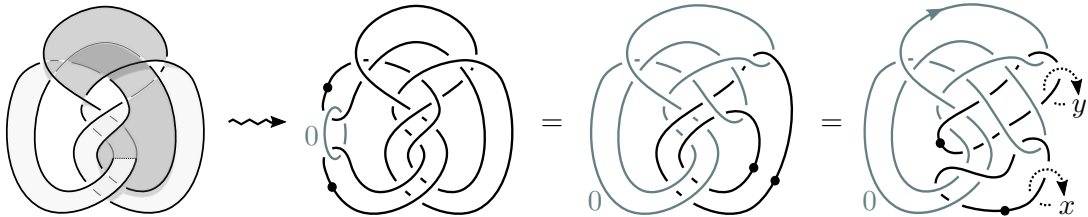\caption{Handle diagrams for the exterior of the slice disk $D$ in $B^4$,  decorated with generators $x$ and $y$ of its fundamental group.}\label{fig:pi1-D}
\end{figure}

\begin{proof}[Proof of Proposition~\ref{prop:top-iso}]
We begin by writing down a handle diagram for the exterior of $D$ in $B^4$ in Figure~\ref{fig:pi1-D}; see \cite[\S6.2]{GompfStipsicz4}. From the rightmost diagram, obtain a presentation of the disk exterior's fundamental group (where the relator is obtained by expressing the 2-handle's attaching curve as a word in the generators $x$ and $y$):
\begin{align*}
\pi_1(B^4 \setminus D) \cong\left \langle  x,y  \mid  1=yyxxx^{-1}y^{-1}=y^2 xy^{-1} \right \rangle \cong \left\langle  x,y  \mid yx=1 \right\rangle \cong \zz. 
\end{align*}
By symmetry, the exterior of $D' \subset B^4$ also has $\pi_1 \cong \zz$,  hence Theorem~\ref{thm:conway-powell} implies that $D$ and $D'$ are topologically isotopic rel boundary.
\end{proof}

However, the above construction  does not \emph{always} yield topologically equivalent disks.

\begin{prop}\label{prop:15n-pi1}
The disks associated to the Mazur cork in Figure~\ref{fig:mazur}  are not topologically isotopic rel boundary.
\end{prop}

\begin{proof} We consider the exterior of one of the slice disks from Figure~\ref{fig:mazur}, which is represented by the handle diagrams shown in Figure~\ref{fig:15n-pi1}. Observe that the loop $\alpha \subset S^3$ does not run over any 1-handles, hence it is nullhomotopic in the disk exterior. An analogous argument shows that the curve $\alpha'$ is nullhomotopic in the exterior of the other slice disk. To show that the disks are not topologically isotopic rel boundary, it suffices to show that $\alpha'$ is not nullhomotopic in the first slice disk exterior.  To that end, we observe that $\pi_1$ has a presentation with two generators $x$ and $y$ and a single relator $x^2yx^{-2} y^{-1}x y$. Here  the loop $\alpha'$ is represented by the word $xy$. In a two-generator group with one (nonempty, cyclically reduced) relator $R$, no proper subword of $R$ is a relator \cite[Theorem 2]{weinbaum}. Since $xy$ is a subword of the (nonempty, cyclically reduced) relator $R=x^2yx^{-2} y^{-1}x y$, the element $xy$ representing $\alpha'$  is nontrivial in $\pi_1$. 
\end{proof}

\begin{figure}[h!]\center
\vspace{-.1in}
\def\svgwidth{.95\linewidth}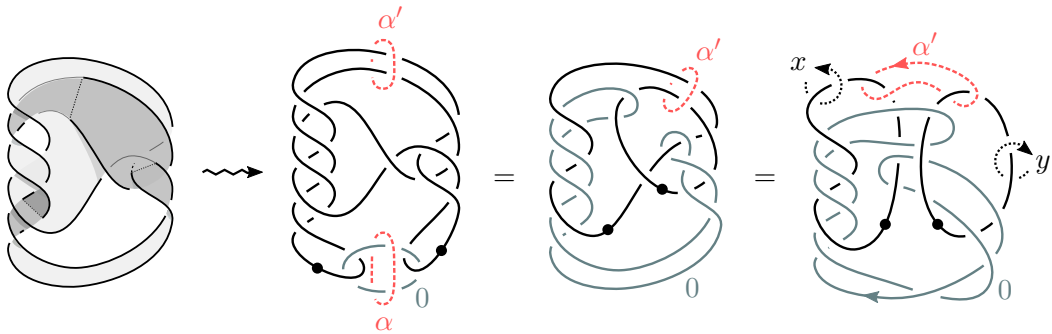\caption{Handle diagrams for the slice disk exterior in Proposition~\ref{prop:15n-pi1}}
\label{fig:15n-pi1}
\end{figure}

\subsection{An infinite family of examples.} Consider the left side of Figure~\ref{fig:positron-family}, which depicts an infinite family of contractible 4-manifolds $W_n$ with $n \geq 0$ (so that each box contains  \emph{left}-handed full twists). The positron cork corresponds to the case $n=0$. It is straightforward to use the methods of \cite{akbulut-matveyev} to show that each $W_n$ is a cork, where the involution exchanges the dotted and zero-labeled components. (In fact, a stronger conclusion is established in \cite[Theorem~1.15]{dhm:corks}: this involution does not extend smoothly over \emph{any} homology ball bounded by $\partial W_n$.) The boundary $\partial W_n$ is the double branched cover of the knot $K_n$ showed in the third frame of Figure~\ref{fig:positron-family}. Each knot  $K_n$  bounds a slice disk $D_n$ analogous to the one shown on the right side of Figure~\ref{fig:invertible}.

\begin{figure}\center
\def\svgwidth{\linewidth}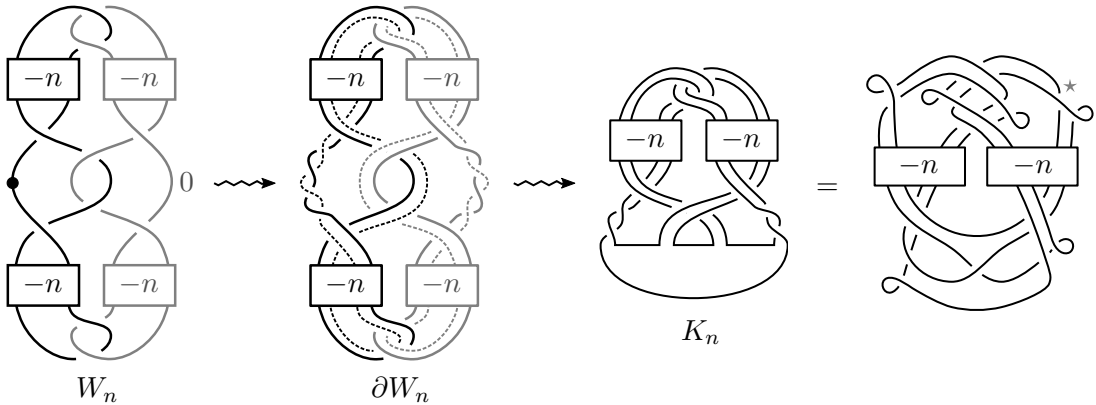\caption{An infinite family of examples generalizing the positron cork.}\label{fig:positron-family}
\end{figure}

\begin{proof}[Proof of Theorem~\ref{thm:cork}]
Fix $n \geq 0$ and let $W_n$, $K_n$, and $D_n$ be as above. To show that $W_n$ is the double branched cover of the pair $(B^4,D_n)$, we first draw and then simplify a handle diagram for the exterior of $D_n \subset B^4$; see Figure~\ref{fig:dbc-family}. For later use, we track a pair of curves $\gamma$ and $\gamma'$ in $S^3 \setminus K_n$. From here, we produce a handle diagram for the double branched cover $\Sigma(B^4,D_n)$ in part (a) of Figure~\ref{fig:dbc-family}, while the remaining steps of Figure~\ref{fig:dbc-family} show that $\Sigma(B^4,D_n)$ is diffeomorphic to the cork $W_n$.

Now observe that there is an involution of the pair $(S^3,K_n)$ that exchanges the two curves $\gamma$ and $\gamma'$ in Figure~\ref{fig:Dn-exterior}(a). A pair of lifts $\tilde \gamma$ and $\tilde \gamma'$ of these curves are shown in part (a) of Figure~\ref{fig:dbc-family}, and we see that these lifts are isotopic to meridians of the 1- and 2-handle curves shown in part (e) of the figure. Therefore the aforementioned symmetry of $(S^3,K_n)$ lifts to the cork-twisting involution of $\partial W_n$ that exchanges the roles of the dotted and 0-labeled curves, as desired.

\begin{figure}\center
\smallskip
\def\svgwidth{.925\linewidth}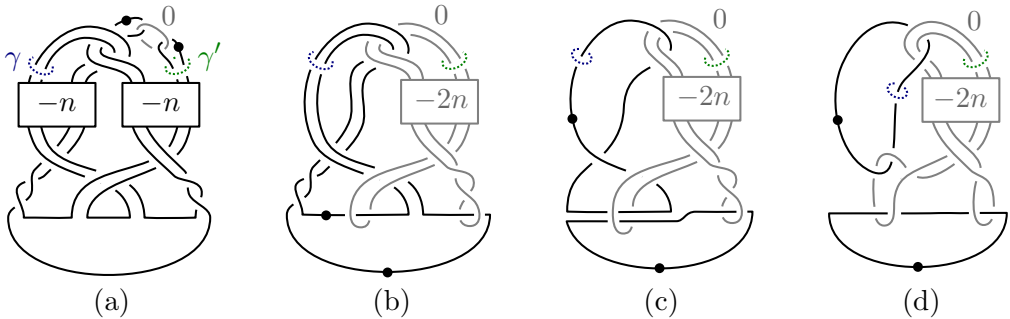\caption{Handle diagrams for the exterior of $D_n \subset B^4$; (c) is obtained from (b) by sliding one 1-handle over the other, and the remaining steps are isotopies.  }\label{fig:Dn-exterior}
\end{figure}

\begin{figure}\center
\def\svgwidth{\linewidth}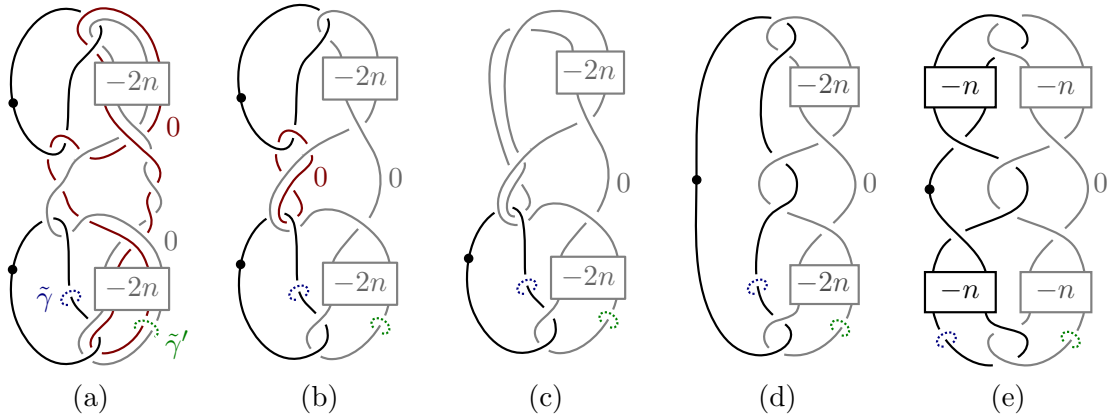\caption{Handle diagrams for the double branched cover $\Sigma(B^4,D_n)$; the first step is a 2-handle slide, the second consists of a 2-handle slide followed by a 1-/2-handle cancellation, and the remaining steps are isotopies.}\label{fig:dbc-family}
\end{figure}

When $n=0$, the disk $D_n$ is the disk $D$ from Figure~\ref{fig:positron}, which we show is isotopic to a compact piece of a complex curve in Theorem~\ref{thm:braid} using positively braided surfaces. For $n \geq 1$, we do not give an explicit description of the disk $D_n$ as a positively braided surface. However, the interested reader can give an alternate (but less explicit) construction using the arguments from \S3-4 of \cite{hayden:disks}. We outline the argument: The right side of Figure~\ref{fig:positron-family} depicts the knot $K_n$ as a transverse knot in the standard contact 3-sphere. Performing a symplectic band move (as in \cite[Lemma 3.2]{hayden:disks}) at the upper-rightmost crossing (labeled with a star) yields a two-component transverse unlink that can be capped off with a pair of symplectic disks; this yields a symplectic disk in the standard symplectic 4-ball bounded by $K_n$. From here, we may combine the work of  Rudolph \cite{rudolph:qp-alg,rudolph:braided-surface} and Boileau-Orevkov \cite{bo:qp} to upgrade the symplectic disk to a compact piece of an algebraic curve in $\cc^2$. 
\end{proof}

\vspace{-.375in}

\section{Exotic surfaces in the open 4-ball}\label{sec:ball}

The goal of this section is to prove Theorem~\ref{thm:ball}. The theorem's focus on surfaces other than the disk reflects the fact that our argument will involve studying  second homology classes in the double branched covers of $B^4$ over our exotic surfaces. For disks in $B^4$, the double branched cover is a (rational) homology 4-ball, hence contains no (non-torsion) homologically essential surfaces. This in mind, we will attach bands to the original disks $D$ and $D'$ to produce larger surfaces whose double branched covers have nontrivial (and non-torsion) second homology.

\begin{figure}[b!]\center
\smallskip
\def\svgwidth{\linewidth}
\begingroup%
  \makeatletter%
  \providecommand\color[2][]{%
    \errmessage{(Inkscape) Color is used for the text in Inkscape, but the package 'color.sty' is not loaded}%
    \renewcommand\color[2][]{}%
  }%
  \providecommand\transparent[1]{%
    \errmessage{(Inkscape) Transparency is used (non-zero) for the text in Inkscape, but the package 'transparent.sty' is not loaded}%
    \renewcommand\transparent[1]{}%
  }%
  \providecommand\rotatebox[2]{#2}%
  \newcommand*\fsize{\dimexpr\f@size pt\relax}%
  \newcommand*\lineheight[1]{\fontsize{\fsize}{#1\fsize}\selectfont}%
  \ifx\svgwidth\undefined%
    \setlength{\unitlength}{1537.36395408bp}%
    \ifx\svgscale\undefined%
      \relax%
    \else%
      \setlength{\unitlength}{\unitlength * \real{\svgscale}}%
    \fi%
  \else%
    \setlength{\unitlength}{\svgwidth}%
  \fi%
  \global\let\svgwidth\undefined%
  \global\let\svgscale\undefined%
  \makeatother%
  \begin{picture}(1,0.20810155)%
    \lineheight{1}%
    \setlength\tabcolsep{0pt}%
    \put(0,0){\includegraphics[width=\unitlength,page=1]{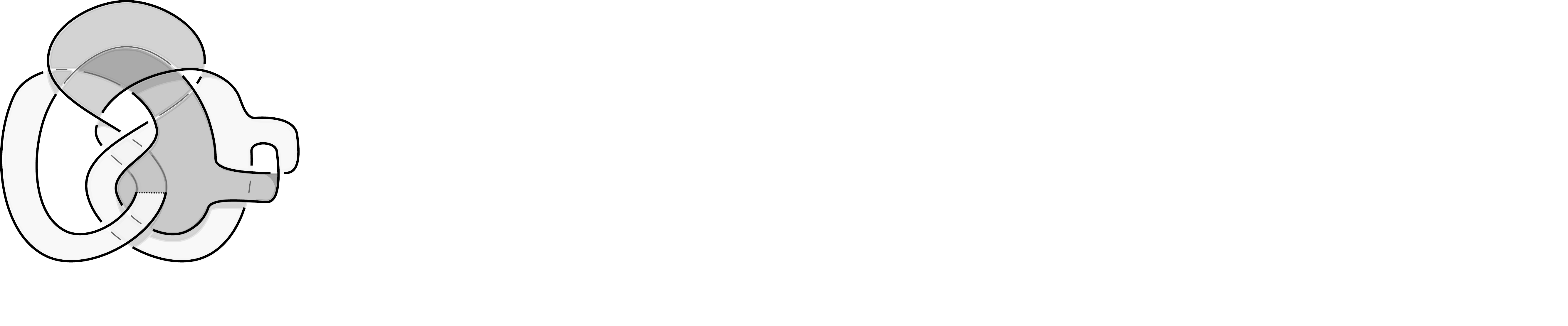}}%
    \put(0.07640009,0.00343976){\makebox(0,0)[lt]{\lineheight{1.25}\smash{\begin{tabular}[t]{l}$A$\end{tabular}}}}%
    \put(0,0){\includegraphics[width=\unitlength,page=2]{building_blocks.pdf}}%
    \put(0.30922891,0.00343976){\makebox(0,0)[lt]{\lineheight{1.25}\smash{\begin{tabular}[t]{l}$A'$\end{tabular}}}}%
    \put(0,0){\includegraphics[width=\unitlength,page=3]{building_blocks.pdf}}%
    \put(0.57366412,0.00343976){\makebox(0,0)[lt]{\lineheight{1.25}\smash{\begin{tabular}[t]{l}$T$\end{tabular}}}}%
    \put(0.85783405,0.00343976){\makebox(0,0)[lt]{\lineheight{1.25}\smash{\begin{tabular}[t]{l}$T'$\end{tabular}}}}%
    \put(0,0){\includegraphics[width=\unitlength,page=4]{building_blocks.pdf}}%
  \end{picture}%
\endgroup%
\caption{Annuli and tori obtained by adding adding bands to the disks $D$ and $D'$.}\label{fig:building-blocks}
\end{figure}

\subsection{Building blocks and their branched covers.} Consider the  annuli $A$ and $A'$ and the tori $T$ and $T'$ shown in Figure~\ref{fig:building-blocks}, which are obtained from the disks $D$ and $D'$ by attaching bands to $K=\partial D=\partial D'$. Since $D$ and $D'$ are topologically isotopic rel boundary, it follows that $A$ and $A'$ are topologically isotopic rel boundary, as are $T$ and $T'$. However, we will show that $A$ and $A'$ (and $T$ and $T'$)   are \emph{not} smoothly isotopic --- in fact, they are not equivalent under any diffeomorphisms of $B^4$.

\begin{prop}\label{prop:spheres}
The double branched covers of $B^4$ over $A'$ and $T'$ contain smoothly embedded 2-spheres of square $-2$.
\end{prop}

\begin{proof}
As illustrated in Figure~\ref{fig:hopf}(a), the annulus $A'$ contains the standard Hopf annulus as an embedded subsurface.  The core curve of the Hopf annulus is an unknot that bounds a standard disk as shown in Figure~\ref{fig:hopf}(b); the disk can be made to intersect the Hopf annulus only along the core curve by pushing the disk's interior into the 4-ball. The disk then lifts to a smoothly embedded 2-sphere of square $-2$ in the double branched cover of $B^4$ over the Hopf annulus, hence into the double branched cover of $A'$ as well. The same argument applies to $T'$.
\end{proof}

\begin{figure}\center
\def\svgwidth{.66\linewidth}
\begingroup%
  \makeatletter%
  \providecommand\color[2][]{%
    \errmessage{(Inkscape) Color is used for the text in Inkscape, but the package 'color.sty' is not loaded}%
    \renewcommand\color[2][]{}%
  }%
  \providecommand\transparent[1]{%
    \errmessage{(Inkscape) Transparency is used (non-zero) for the text in Inkscape, but the package 'transparent.sty' is not loaded}%
    \renewcommand\transparent[1]{}%
  }%
  \providecommand\rotatebox[2]{#2}%
  \newcommand*\fsize{\dimexpr\f@size pt\relax}%
  \newcommand*\lineheight[1]{\fontsize{\fsize}{#1\fsize}\selectfont}%
  \ifx\svgwidth\undefined%
    \setlength{\unitlength}{940.2692179bp}%
    \ifx\svgscale\undefined%
      \relax%
    \else%
      \setlength{\unitlength}{\unitlength * \real{\svgscale}}%
    \fi%
  \else%
    \setlength{\unitlength}{\svgwidth}%
  \fi%
  \global\let\svgwidth\undefined%
  \global\let\svgscale\undefined%
  \makeatother%
  \begin{picture}(1,0.4508199)%
    \lineheight{1}%
    \setlength\tabcolsep{0pt}%
    \put(0,0){\includegraphics[width=\unitlength,page=1]{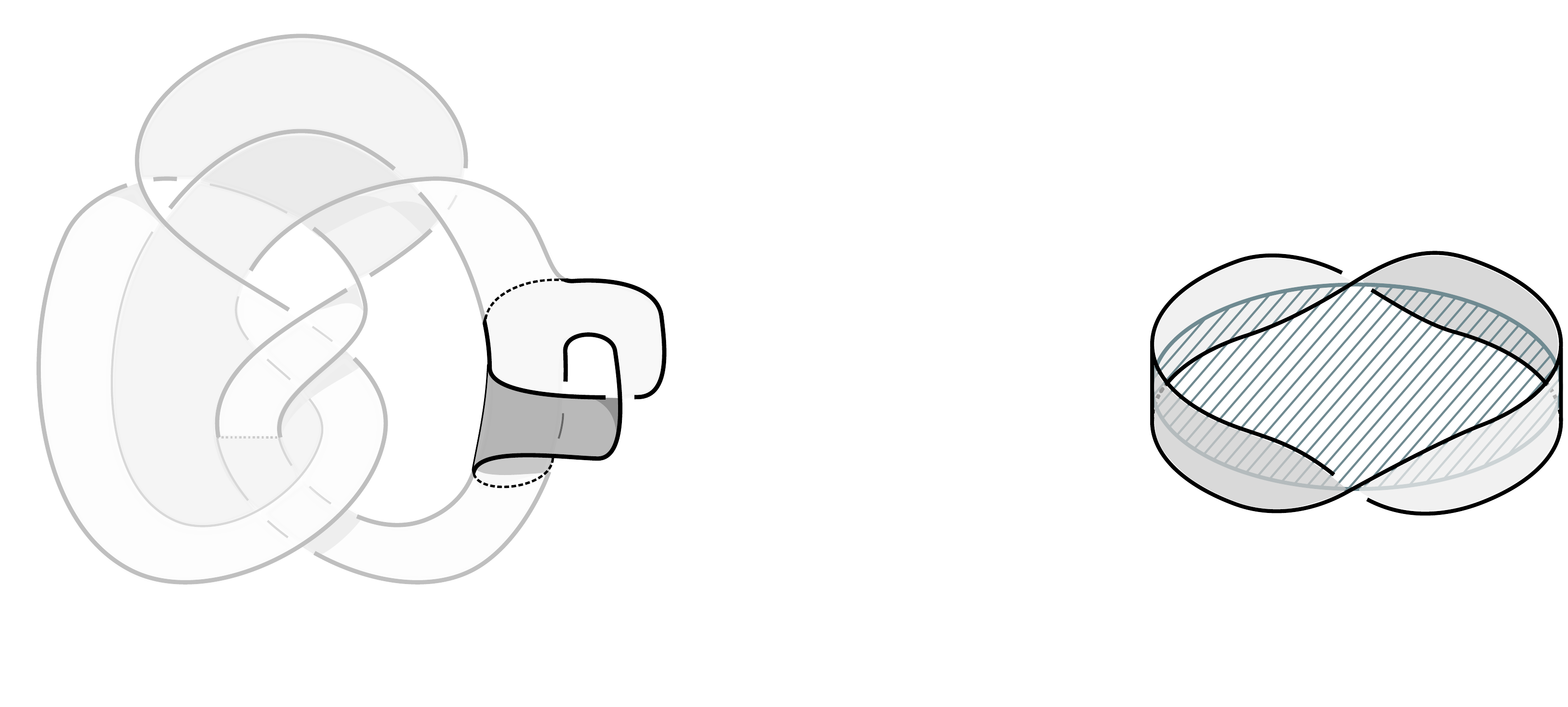}}%
    \put(0.18482555,0.00322787){\makebox(0,0)[lt]{\lineheight{1.25}\smash{\begin{tabular}[t]{l}(a)\end{tabular}}}}%
    \put(0.84573033,0.00322787){\makebox(0,0)[lt]{\lineheight{1.25}\smash{\begin{tabular}[t]{l}(b)\end{tabular}}}}%
  \end{picture}%
\endgroup%

\caption{Locating a Hopf annulus inside $A'$, as well as a disk that lifts to a 2-sphere of square $-2$ in the double branched cover of $B^4$ along the Hopf annulus.}\label{fig:hopf}
\end{figure}

\begin{prop}\label{prop:no-spheres}
The double branched covers of $B^4$ over $A$ and $T$ do not contain smoothly embedded 2-spheres of square $-2$.
\end{prop}

To prove the proposition, we first show that these double branched covers admit simple handle decompositions.

\begin{lem}\label{lem:DBCs}
The double branched covers of $B^4$ over $A$ and $T$ are represented by the handle diagrams shown in Figure~\ref{fig:DBCs}.
\end{lem}

We postpone the proof until we discuss braided surfaces in \S\ref{subsec:revisit}. The attaching curves in Figure~\ref{fig:DBCs} are drawn as Legendrian knots with respect to the standard contact structure on $S^3$. Moreover, for each attaching curve $\mathcal{K}_i$, the 2-handle framing is one less than the  Thurston-Bennequin number $tb(\mathcal{K}_i)$, which can be computed diagrammatically as
$$tb(\mathcal{K}_i)= \mathrm{writhe}- \# \mathrm{\ right \  cusps}.$$ 
By \cite{yasha:stein,yasha:knots} (cf., \cite{GompfStipsicz4,gompf:stein}), the resulting 4-manifolds each admit a corresponding Stein structure.  Therefore we may appeal to a general constraint on embedded surfaces in Stein domains due to Lisca and Mati\'c.

\begin{figure}\center
\smallskip
\def\svgwidth{.66\linewidth}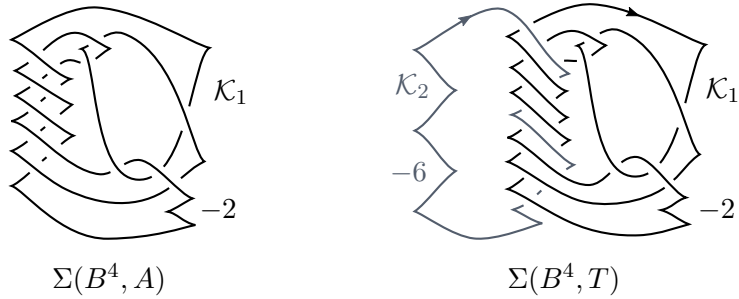
\caption{Stein handle diagrams for the double branched covers of $B^4$ over $A$ and $T$.}\label{fig:DBCs}
\end{figure}

\begin{thm}[\cite{lisca-matic}]\label{thm:lisca-matic-2}
If $S$ is a smoothly embedded surface in a Stein domain $W$ that represents a nontrivial class in $H_2(W)$, then
$$[S]\cdot [S] + \left| \langle c_1(W) , [S]\rangle\right | \leq 2g(S)-2.$$
\end{thm}

To wield this obstruction, we require two additional preliminaries. Recall that if a 4-manifold $W$ is obtained from $B^4$ by attaching 2-handles along a framed link $L$, then $H_2(W) \cong \zz^{\ell}$, where $\ell$ is the number of components of $L$. Given an orientation on $L$, there is a natural basis $h_1, \ldots, h_\ell$ for $H_2(W)$, where $h_i$ is represented by an oriented surface obtained by capping off a Seifert surface for the link component $\mathcal{K}_i$ of $L$ with the core of the 2-handle attached along $\mathcal{K}_i$.  Moreover, with respect to this basis, the intersection form $Q_W$ on $H_2(W)$ is given by the linking matrix of the framed link $L$.

Given a Stein handle diagram for $W$ (as described above),  we can also understand the first Chern class $c_1(W)$ of the associated complex structure on $W$ using the link $L$  \cite[Proposition~2.3]{gompf:stein}. In particular, with $h_i$ defined as above,
\begin{equation}\label{eqn:chern}
\langle c_1(W), h_i \rangle  = r(\mathcal{K}_i),
\end{equation}
where $r$ denotes the rotation number of the oriented Legendrian knot $\mathcal{K}_i$. The rotation number can be calculated diagrammatically as the number of left cusps at which $\mathcal{K}_i$   oriented downwards \emph{minus} the number of right cusps at which $\mathcal{K}_i$ is oriented upwards.

\begin{proof}[Proof of Proposition~\ref{prop:no-spheres}]
By Lemma~\ref{lem:DBCs} and Figure~\ref{fig:DBCs}, the 4-manifold $\Sigma(B^4,A)$ is obtained from $B^4$ by attaching a $(-2)$-framed 2-handle along a Legendrian knot $\mathcal{K}_1$ with $tb(\mathcal{K}_1)=-1$ and $r(\mathcal{K}_1)=-2$. Therefore $\Sigma(B^4,A)$ admits a Stein structure, and \eqref{eqn:chern} implies that any homologically essential surface $S$ in $\Sigma(B^4,A)$ satisfies $|\langle c_1(\Sigma(B^4,A)),[S] \rangle | \neq 0$. On the other hand, if $S$ is a smoothly embedded 2-sphere of square $-2$ in $\Sigma(B^4,A)$, then Theorem~\ref{thm:lisca-matic-2} implies that $|\langle c_1(\Sigma(B^4,A)) , [S] \rangle | = 0$. It follows that $\Sigma(B^4,A)$ contains no smoothly embedded 2-spheres of square $-2$.

Next we consider $\Sigma(B^4,T)$. Using the handle diagram from Figure~\ref{fig:DBCs}, we see that $H_2(\Sigma(B^4,T)) \cong \zz^2$ has generators $h_1, h_2$ corresponding to the link components $\mathcal{K}_1$ and $\mathcal{K}_2$. The intersection form for $\Sigma(B^4,T)$ is $$Q_{\Sigma(B^4,T)} =\begin{bmatrix} -2 & -2 \\ -2 & -6 \end{bmatrix}$$
with respect to this basis. The underlying Legendrian link components have $r(\mathcal{K}_1)=-2$ and $r(\mathcal{K}_2)=0$, so $\langle c_1(\Sigma(B^4,T)), h_1 \rangle = -2$ and $\langle c_1(\Sigma(B^4,T)), h_2 \rangle =0$.

Up to multiplication by $\pm 1$, there is a unique class $\sigma \in H_2(\Sigma(B^4,T))$ of square $-2$; indeed, if $\sigma = c_1 h_1 + c_2 h_2$, then
\begin{align*}
\sigma \cdot \sigma &= c_1(-2c_1 - 2c_2) +c_2(-2c_1-6c_2)
= -2(c_1^2+2c_1c_2 + 3c_2^2),
\end{align*}
and it is  straightforward to see that solving $\sigma \cdot  \sigma=-2$ requires $c_1=\pm 1$ and $c_2=0$. However, $|\langle c_1(\Sigma(B^4,T)),  \pm h_1 \rangle| = 2$, so Theorem~\ref{thm:lisca-matic-2} implies that $\pm h_1$ cannot be represented by a smoothly embedded 2-sphere. It follows that $\Sigma(B^4,T)$ does not contain any smoothly embedded 2-spheres of self-intersection $-2$.
\end{proof}

\begin{figure}\center
\smallskip
\def\svgwidth{\linewidth}
\begingroup%
  \makeatletter%
  \providecommand\color[2][]{%
    \errmessage{(Inkscape) Color is used for the text in Inkscape, but the package 'color.sty' is not loaded}%
    \renewcommand\color[2][]{}%
  }%
  \providecommand\transparent[1]{%
    \errmessage{(Inkscape) Transparency is used (non-zero) for the text in Inkscape, but the package 'transparent.sty' is not loaded}%
    \renewcommand\transparent[1]{}%
  }%
  \providecommand\rotatebox[2]{#2}%
  \newcommand*\fsize{\dimexpr\f@size pt\relax}%
  \newcommand*\lineheight[1]{\fontsize{\fsize}{#1\fsize}\selectfont}%
  \ifx\svgwidth\undefined%
    \setlength{\unitlength}{3073.3739081bp}%
    \ifx\svgscale\undefined%
      \relax%
    \else%
      \setlength{\unitlength}{\unitlength * \real{\svgscale}}%
    \fi%
  \else%
    \setlength{\unitlength}{\svgwidth}%
  \fi%
  \global\let\svgwidth\undefined%
  \global\let\svgscale\undefined%
  \makeatother%
  \begin{picture}(1,0.21742546)%
    \lineheight{1}%
    \setlength\tabcolsep{0pt}%
    \put(0,0){\includegraphics[width=\unitlength,page=1]{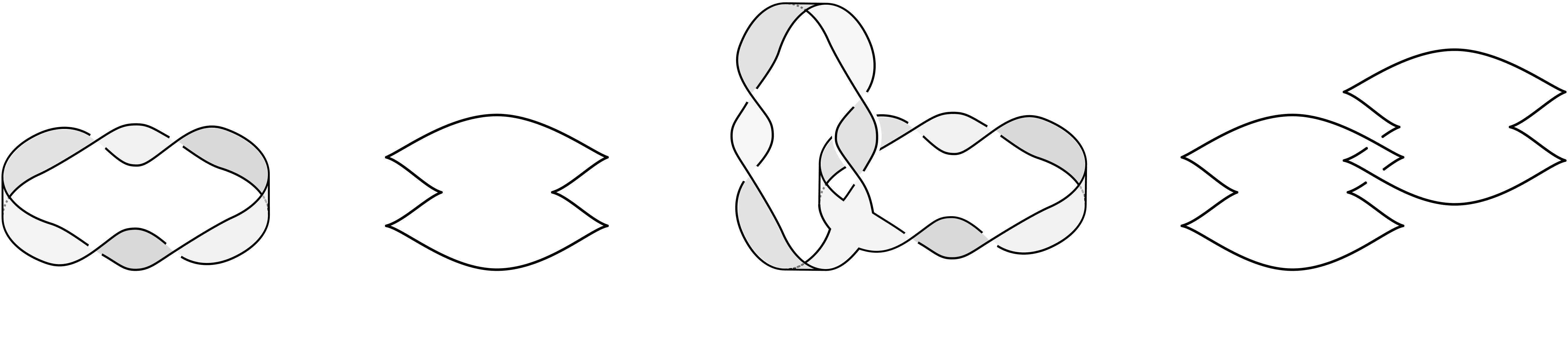}}%
    \put(0.07157473,0.00234446){\makebox(0,0)[lt]{\lineheight{1.25}\smash{\begin{tabular}[t]{l}$A_0$\end{tabular}}}}%
    \put(0.26135876,0.00234446){\makebox(0,0)[lt]{\lineheight{1.25}\smash{\begin{tabular}[t]{l}$\Sigma(B^4,A_0)$\end{tabular}}}}%
    \put(0.56528233,0.00234446){\makebox(0,0)[lt]{\lineheight{1.25}\smash{\begin{tabular}[t]{l}$T_0$\end{tabular}}}}%
    \put(0.82678836,0.00234446){\makebox(0,0)[lt]{\lineheight{1.25}\smash{\begin{tabular}[t]{l}$\Sigma(B^4,T_0)$\end{tabular}}}}%
    \put(0.2449578,0.14934847){\makebox(0,0)[lt]{\lineheight{1.25}\smash{\begin{tabular}[t]{l}$-4$\end{tabular}}}}%
    \put(0.75334022,0.14529755){\makebox(0,0)[lt]{\lineheight{1.25}\smash{\begin{tabular}[t]{l}$-4$\end{tabular}}}}%
    \put(0.85194814,0.1899914){\makebox(0,0)[lt]{\lineheight{1.25}\smash{\begin{tabular}[t]{l}$-4$\end{tabular}}}}%
  \end{picture}%
\endgroup%

\caption{The accessory surfaces $A_0$ and $T_0$, along with Stein handle diagrams for their double branched covers.}\label{fig:more-plumbing}
\end{figure}
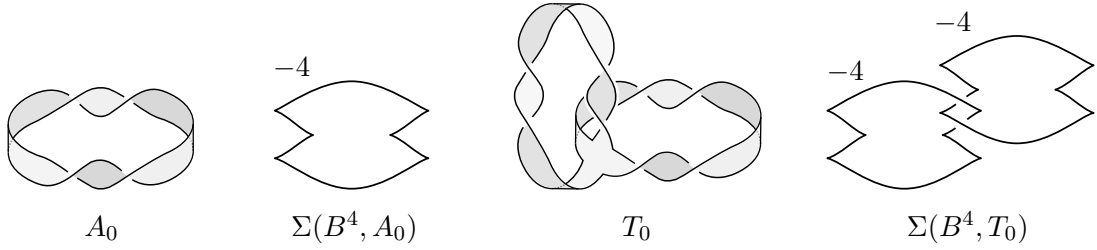

\subsection{Exotic surfaces in the open 4-ball.} We now combine these ingredients  to produce exotic surfaces in $B^4$ whose interiors in the open 4-ball are distinct.

\begin{proof}[Proof of Theorem~\ref{thm:ball}]
We begin by introducing two last building blocks, the \emph{accessory annulus}  $A_0$ and \emph{accessory torus} $T_0$ shown in Figure~\ref{fig:more-plumbing}. It is straightforward to show that the double branched cover of $B^4$ along $A_0$ is a $D^2$-bundle over $S^2$ with Euler number $-4$. This  can be obtained by attaching a 2-handle to $B^4$ along a $(-4)$-framed unknot. It follows that $\Sigma(B^4,A_0)$ admits a Stein handle diagram as shown in Figure~\ref{fig:more-plumbing}. The accessory torus $T_0$ is obtained by plumbing two copies of $A_0$ together. The double branched cover $\Sigma(B^4,T_0)$ is then a plumbing of two $D^2$-bundles over $S^2$ each with Euler number $-4$; see  \cite[\S6.1]{GompfStipsicz4}. As illustrated using the Stein handle diagram on the right side of Figure~\ref{fig:more-plumbing}, this 4-manifold also admits a Stein structure. 

We have already established two key cases of the theorem, namely the cases of annuli and once-punctured tori, using Propositions~\ref{prop:spheres} and \ref{prop:no-spheres}. With these pieces in hand, the general case follows easily. Other than the disk, the topological type of any compact, connected, orientable surface with boundary can be realized by boundary-summing copies of annuli and once-punctured tori. In particular, surfaces  of positive genus can be realized by boundary-summing the tori $T$ or $T'$ with appropriate copies of the accessory surfaces $A_0$ and/or $T_0$, whereas planar surfaces are realized as boundary sums of the annuli $A$ or $A'$ with copies of $A_0$. In each case, let $F$ and $F'$ denote the resulting surfaces with the desired topological type.   Note that the topological isotopy (rel boundary) between the annuli $A$ and $A'$ or the tori $T$ and $T'$ induces a topological isotopy (rel boundary) between $F$ and $F'$.

As in Proposition~\ref{prop:spheres}, the double branched cover $\Sigma(B^4,F')$ will contain a smooth 2-sphere of self-intersection $-2$.  On the other hand, we claim that $W=\Sigma(B^4,F)$ cannot contain any such 2-sphere. We first treat the case where $F$ has positive genus. In this case, $F$ is a boundary sum of $T$ with some number of copies of $A_0$ and $T_0$. To be concrete, we write
$$F = T \, \natural_a \ \! A_0  \, \natural_b \ \! T_0$$
for some integers $a,b \geq 0$. It follows that the branched cover  $W=\Sigma(B^4,F)$ is obtained by boundary-summing (via 4-dimensional 1-handles) the corresponding branched covers of $B^4$ over these constituent surfaces:
\begin{equation}\label{eqn:decomp}
W = \Sigma(B^4,T) \, \natural_a \ \! \Sigma(B^4,A_0)   \, \natural_b \ \! \Sigma(B^4,T_0).
\end{equation}
The intersection form of $W$ decomposes  as a direct sum:
$$
\begin{bmatrix} -2 & -2 \\ -2 & -6 \end{bmatrix} \  \ \underset{\text{\normalsize \ \  $a$}}{\text{\Large$\oplus$}} \, \begin{bmatrix} -4 \end{bmatrix} \ \ \underset{\text{\normalsize  \ \  $b$}}{\text{\Large$\oplus$}}  \, \begin{bmatrix} -4 & -1 \\ -1 & -4\end{bmatrix}.
$$ 
As above, it is straightforward to see that $H_2(W)$ contains a unique (up to multiplication by $\pm1$) class $\sigma$ of self-intersection $-2$. This corresponds to the class $h_1$ described above in the  summand $\Sigma(B^4,T) \subset W$; this is represented by the $-2$-framed handle attached along $\mathcal{K}_1$ on the right side Figure~\ref{fig:DBCs}.  We claim that $\sigma$ cannot be represented by a smoothly embedded 2-sphere. By choosing to connect the summands of $W$ in \eqref{eqn:decomp} via Stein 1-handles, we obtain a Stein structure on $W$ whose first Chern class $c_1(W)$ decomposes a direct sum of the summands' first Chern classes. Thus we have
$$|\langle c_1(W),\sigma \rangle| = |\langle c_1(\Sigma(B^4,T)), \sigma \rangle| = 2 \neq 0.$$
By Theorem~\ref{thm:lisca-matic-2}, it follows that $\sigma$ cannot be represented by a smoothly embedded 2-sphere. Therefore $W$ contains no smooth 2-spheres of self-intersection $-2$.
 
This completes the argument for the case of surfaces of positive genus. In the simpler planar case, $F$ is obtained by boundary summing $A$ with copies of $A_0$. The double branched cover $W = \Sigma(B^4,F)$ then has the (diagonal) intersection form $\begin{bmatrix} -2 \end{bmatrix} \oplus \begin{bmatrix} -4 \end{bmatrix} \oplus \cdots \oplus \begin{bmatrix} -4 \end{bmatrix}$. It follows immediately that $H_2(W)$ contains a unique class (up to multiplication by $\pm1$) of self-intersection $-2$; this class is represented by the $(-2)$-framed 2-handle in $\Sigma(B^4,A) \subset W$. As above, $W$ admits a Stein structure whose first Chern class decomposes along the summands. The aforementioned 2-handle is attached along a Legendrian knot with rotation number $-2$, so
$$|\langle c_1(W), \sigma \rangle| = |\langle c_1(\Sigma(B^4,A)), \sigma \rangle | = 2 \neq 0.$$
As above, we conclude that $W$ contains no smooth 2-spheres  of self-intersection $-2$. 

This argument distinguishes the interiors of $W$ and $W'$, so the interiors of the surfaces $F$ and $F'$ remain exotically knotted in the open 4-ball, as desired. 
\end{proof}

\vspace{-.35in}

\section{Braided surfaces and complex curves}
\label{sec:braid}

In this section, we recast the surfaces in $B^4$ from \S\ref{sec:invertible} and \S\ref{sec:ball} as \emph{positively braided surfaces}, enabling us to realize them as compact pieces of algebraic curves.

\subsection{Braided surfaces.}\label{subsec:braid} To introduce positively braided surfaces, we first recall some preliminaries about the $n$-stranded braid group
$$B_n = \big \langle \, \sigma_1,\ldots,\sigma_{n-1} \mid \sigma_i \sigma_{i+1} \sigma_i = \sigma_{i+1} \sigma_i \sigma_{i+1}, \ \sigma_i \sigma_j = \sigma_j \sigma_i \text{ for } |i-j| \geq 2 \, \big \rangle.$$
Recall from the introduction that an element of $B_n$ is \emph{quasipositive} if it is a product of conjugates of the standard positive Artin generators $\sigma_i \in B_n$ \cite{rudolph:qp-alg}. Each quasipositive factorization of a quasipositive $n$-braid  determines a ribbon-immersed surface in $S^3$ formed from $n$ parallel disks by attaching a positively twisted band for each term $w \sigma_i w^{-1}$; see Figure~\ref{fig:braid-and-surface}.  
The corresponding embedded surface in $B^4$ is known as  a  \emph{positively braided surface}, and can naturally be viewed in the bidisk $D^2 \! \times \! D^2$:

\begin{defn}[cf \cite{rudolph:braided-surface}]\label{def:braided}
A smooth, properly embedded surface $F$ in the bidisk $D^2 \! \times \! D^2 = \{(z,w) \in \cc^2 : |z| \leq 1, \ |w| \leq 1\}$  is \emph{positively braided} if 
\begin{enumerate}
\item [(1)] $\partial F$ is a braid in the solid torus $\partial D^2 \! \times \! D^2$,
\item [(2)] the $z$-coordinate projection ${\operatorname{pr}_{z}}|_F: F \to D^2$ is a branched covering, and
\item [(3)] $F$ is oriented so that ${\operatorname{pr}_{z}}|_F$ is orientation-preserving at all regular points and  ${\operatorname{pr}_{w}}|_F$ is orientation-preserving at all branch points (with respect to the complex orientation on $D^2 \subset \cc$).
\end{enumerate}
\end{defn}

 In \cite{rudolph:qp-alg} (as interpreted in \cite[\S4]{rudolph:braided-surface}), Rudolph showed that such surfaces can be realized (up to isotopy) as compact pieces of algebraic curves:
  
\begin{thm}[Rudolph]\label{thm:rudolph}
If $F$ is a positively braided surface in the bidisk $D^2 \! \times \! D^2  \subset \cc^2$, then there is a smooth algebraic curve $V \subset \cc^2$ such that $V \cap D^2 \! \times \! D^2$ is a positively braided surface that is isotopic to $F$ (through positively braided surfaces).
\end{thm}

For later use, we record a fact about positively braided surfaces bounded by unlinks.

\begin{lem}\label{lem:unlink}
Any positively braided surface bounded by an unlink is unique up to isotopy (rel boundary) through smooth surfaces.
\end{lem}

\begin{proof}
Let $F$ be a positively braided surface in $D^2 \! \times \! D^2$ whose boundary is an unlink; as we only care about the smooth isotopy type of $F$ rel boundary, we may smooth the bidisk's corners and view $F$ as a properly embedded surface in $B^4$.

By the  result of Rudolph cited above, $F$ is smoothly isotopic to a compact piece of an algebraic curve $F_0$ in $B^4\subset \cc^2$. (Note that this isotopy does \emph{not} necessarily fix $\partial F$.) By Eliashberg's technique of filling by holomorphic curves (especially \cite[Corollary~2.2B]{yasha:lagr-cyl}), any  smooth, holomorphic curve in $B^4 \subset \cc^2$ bounded by an unlink is smoothly isotopic to a Seifert surface formed from a collection of disjointly embedded disks in $\partial B^4$; see \cite[Proposition~2]{boileau-fourrier} or \cite[\S3.2, p418]{hayden:cross} for  additional details. In particular, there is a collection of embedded disks $F_1 \subset \partial B^4$ with $\partial F_1 = \partial F_0$ such that $F_0$ and $F_1$ are smoothly isotopic (rel boundary).

Now use the isotopy of $B^4$ that carried $F$ to $F_0$ to pull back the isotopy between $F_0$ and $F_1$ (fixing $\partial F_0=\partial F_1$). This furnishes an isotopy from $F$ to a collection of embedded disks in $\partial B^4$ (fixing $\partial F$). We conclude by appealing to the classical fact that any smooth, disjointly embedded collection of disks in $S^3$ is unique up to smooth isotopy (rel boundary).
\end{proof}

\subsection{Revisiting the building blocks.} \label{subsec:revisit} The first goal of this subsection is to recast the disks $D$ and $D'$ as braided surfaces.

\begin{thm}\label{thm:braid}
The disks $D$ and $D'$ are smoothly isotopic to the positively braided surfaces defined by the quasipositive braidwords
\begin{align*}
\beta&=(\sigma_{2}\sigma_{3}\sigma_{2}^{-1})(\sigma_1^{-2} \sigma_2 \sigma_3 \sigma_4^2 \sigma_3^{-1} \sigma_2 \sigma_3 \sigma_4^{-2}  \sigma_3^{-1} \sigma_2^{-1} \sigma_1^2)(\sigma_3^{-1} \sigma_2 \sigma_1 \sigma_2^{-1} \sigma_3)(\sigma_4^{-1} \sigma_3 \sigma_4)\\
\beta'&=\sigma_{2}(w \sigma_{2}^{-1}\sigma_{1}\sigma_{2}w^{-1})(w\sigma_{2}^{-1}\sigma_{3}\sigma_{1}\sigma_{2}\sigma_{1}^{-1}\sigma_{3}^{-1}\sigma_{2}w^{-1}) (w\sigma_{3}^2\sigma_{4}\sigma_{3}^{-2}w^{-1}),
\end{align*}
where $w = \sigma_3 \sigma_4^{-1} \sigma_1^{-1} \sigma_3^{-2} \sigma_2^{-1}   \sigma_1^{-1} \sigma_3^{-1}$. Therefore each of $D$ and $D'$ is smoothly isotopic to the intersection of $B^4$ with a smooth algebraic curve in $\cc^2$. 
\end{thm}

\begin{figure}\center
\smallskip
\def\svgwidth{\linewidth}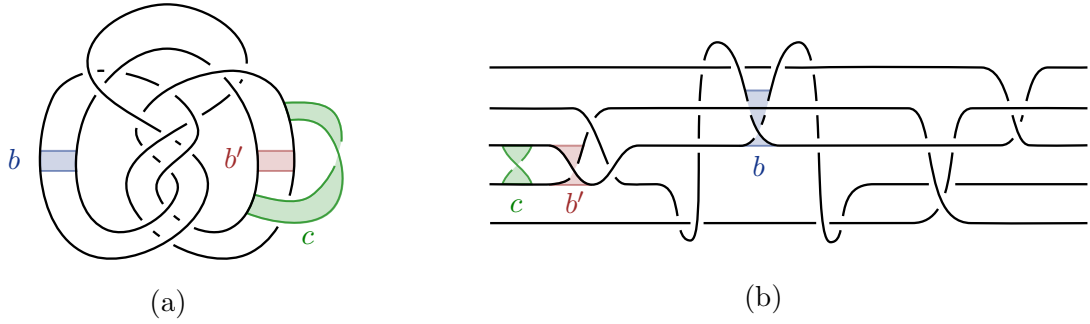
\caption{The knot $K$ is isotopic to the closure of the braid shown in part (b).}\label{fig:braid-to-sym}
\end{figure}

\begin{figure}[b]\center
\includegraphics[width=.35\linewidth]{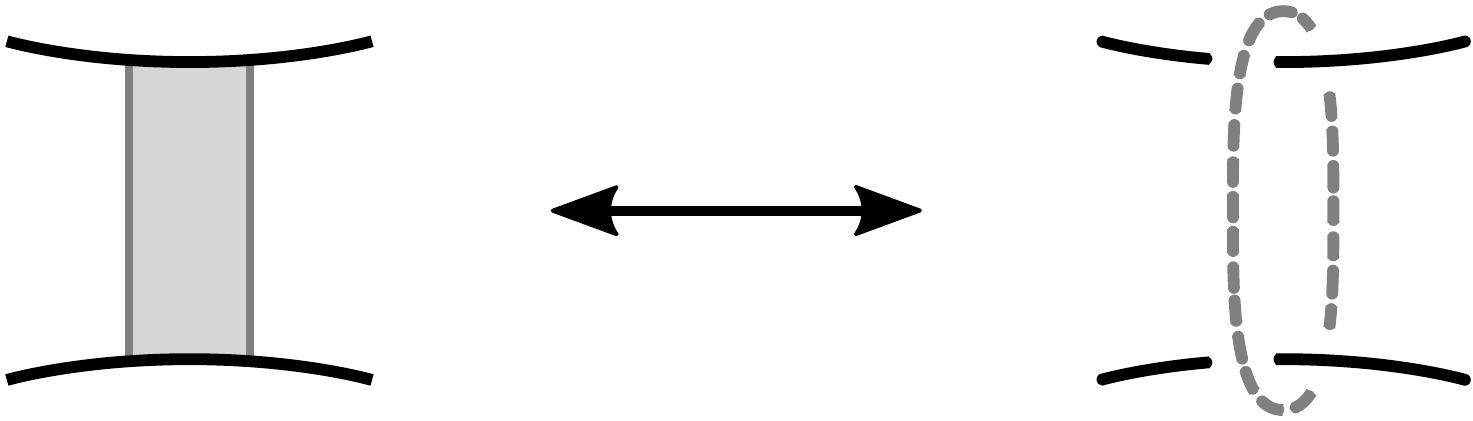}
\caption{Replacing a band with a loop and vice versa.}\label{fig:band-to-loop}
\end{figure}

 We begin by expressing $K$ as the closure of a braid.

 \begin{prop}\label{prop:braid}
The knot $K$ is isotopic to the closure of the braid  
$$\beta=(\sigma_{2}\sigma_{3}\sigma_{2}^{-1})(\sigma_1^{-2} \sigma_2 \sigma_3 \sigma_4^2 \sigma_3^{-1} \sigma_2 \sigma_3 \sigma_4^{-2}  \sigma_3^{-1} \sigma_2^{-1} \sigma_1^2)(\sigma_3^{-1} \sigma_2 \sigma_1 \sigma_2^{-1} \sigma_3)(\sigma_4^{-1} \sigma_3 \sigma_4),$$
as shown in Figure~\ref{fig:braid-to-sym}. Moreover, the isotopy respects the bands $b$, $b'$, and $c$ shown.
\end{prop}

\begin{proof}
This is verified in parts (a)-(z) of Figure~\ref{fig:big-iso-1}, where we begin by drawing the closure of the braid $\beta$ and work towards the desired diagram of $K$. Note that, rather than drawing the bands themselves, we find it easier to follow a collection of tight loops from which the bands can be recovered as in Figure~\ref{fig:band-to-loop}. (The loop uniquely determines a band provided that no other strands pass through the loop during the isotopy.)
\end{proof}

\begin{proof}[Proof of Theorem~\ref{thm:braid}]
We first recall that a slice disk in $B^4$ with two minima and one saddle can be determined (up to isotopy) by a choice of a band move that turns the boundary knot into a two-component unlink. In particular, the disks $D$ and $D'$ bounded by the knot $K$ are determined by the bands $b$ and $b'$ shown in Figure~\ref{fig:braid-to-sym}(a).

\begin{figure}

\hspace{-.05\linewidth} \def\svgwidth{1.1\linewidth}\input{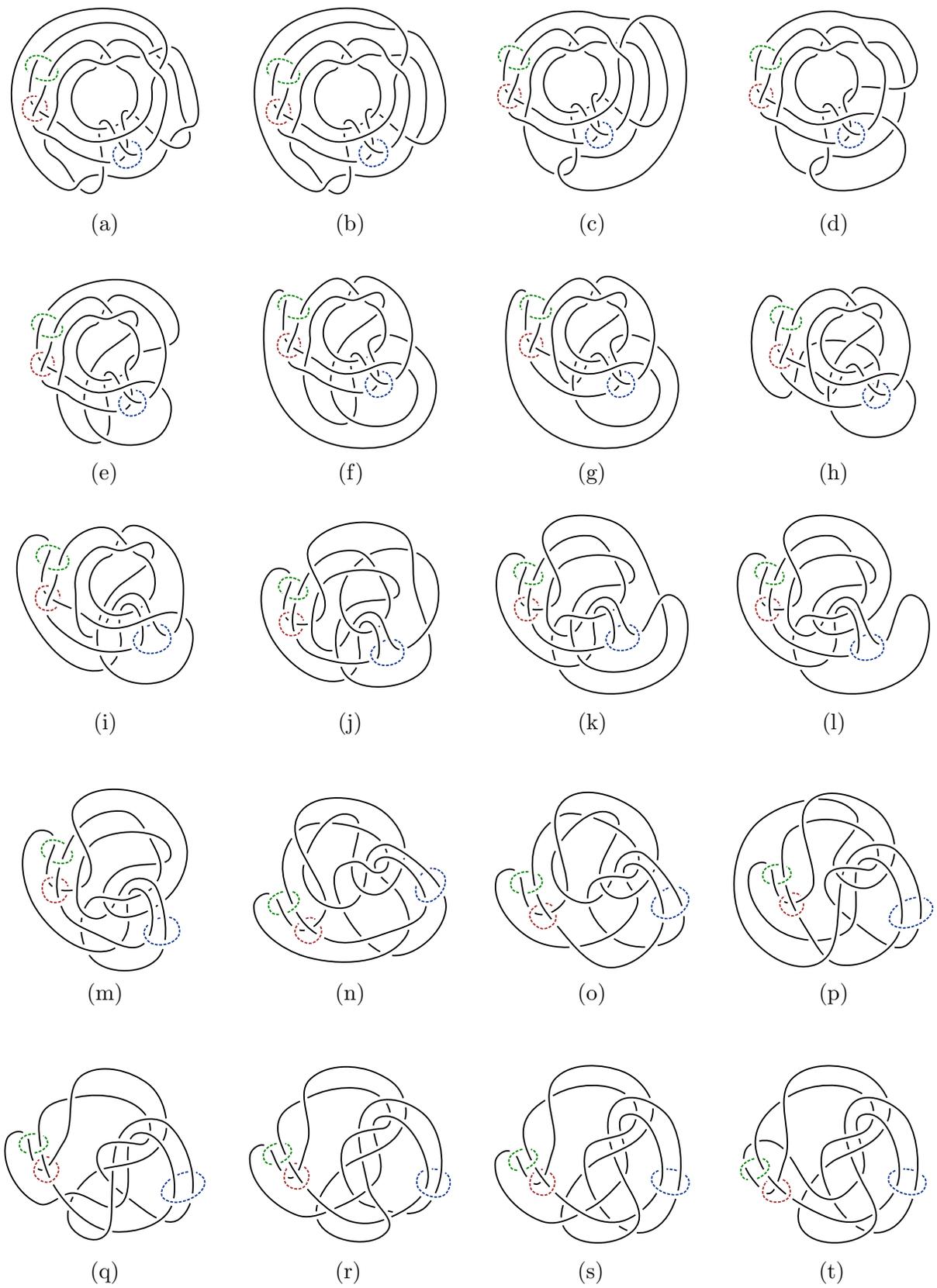}
   \renewcommand\figurename{Figure}
   \renewcommand{\thefigure}{\arabic{figure}}
\caption{An isotopy relating the two diagrams of $K$ from Figure~\ref{fig:braid-to-sym}.}
   \label{fig:big-iso-1}
\end{figure}

\begin{figure}\center
\def\svgwidth{\linewidth}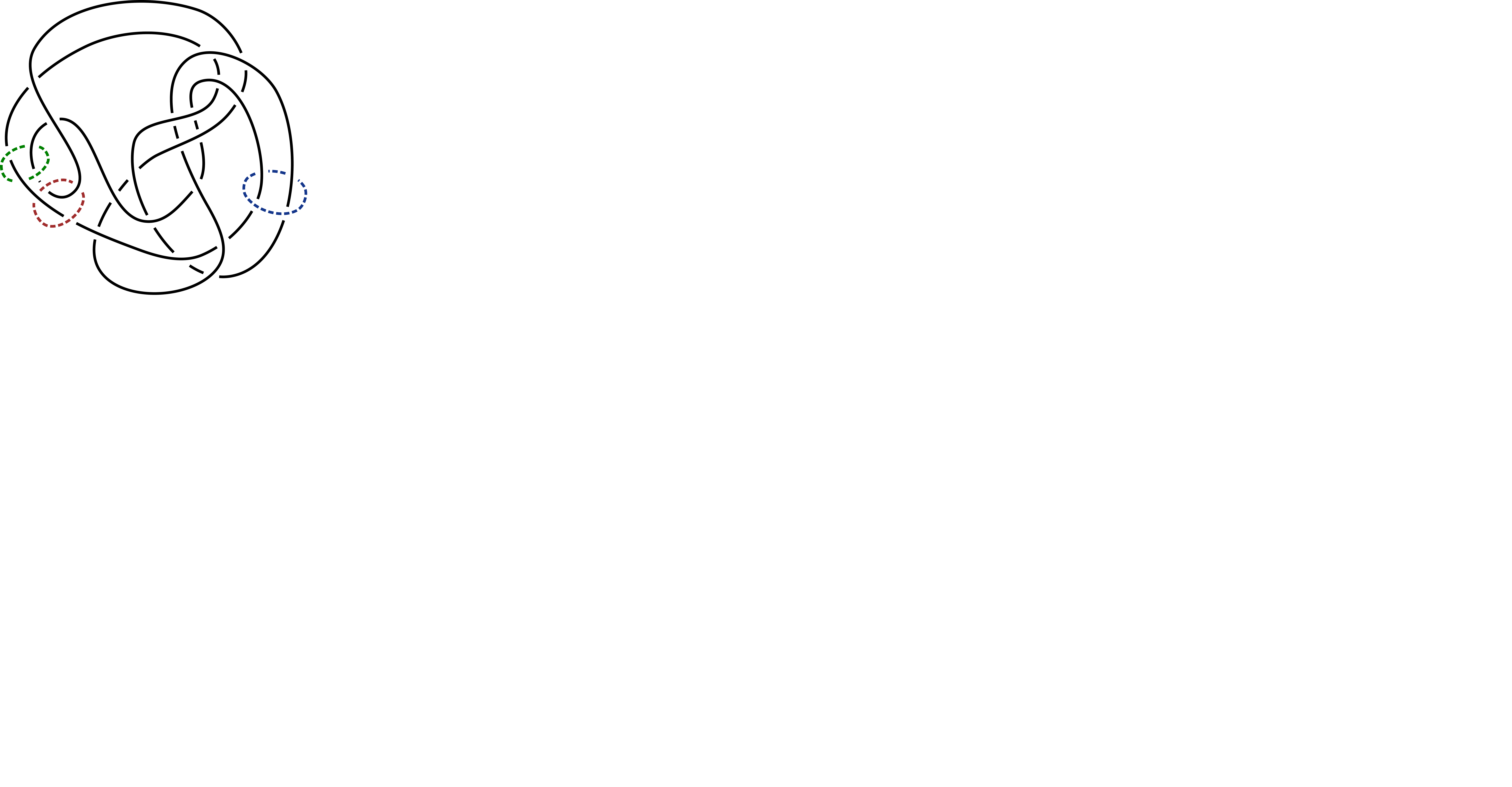
   \renewcommand\figurename{Figure}
   \renewcommand{\thefigure}{\arabic{figure} - Continued}
   \addtocounter{figure}{-1}
\caption{An isotopy relating the two diagrams of $K$ from Figure~\ref{fig:braid-to-sym}.}\label{fig:big-iso-2}
\end{figure}

The proof of Proposition~\ref{prop:braid} shows that the bands $b$ and $b'$ on $K$ are taken to the corresponding bands shown in Figure~\ref{fig:braid-to-sym}(b). It remains to show that the slice disk determined by the band $b$  (resp.~$b'$) in Figure~\ref{fig:braid-to-sym}(b) is smoothly isotopic rel boundary to the slice disk determined by  the braidword $\beta$ (resp.~$\beta'$). Observe that resolving the central crossing in the band $b$  (i.e., deleting the corresponding $\sigma_2 \in \beta$) yields a two-component unlink. The surface determined by $b$ is obtained from the standard Seifert surface for the two-component unlink by reattaching the band $b$. The surface determined by  $\beta$ also contains the band $b$. If we remove this band, the resulting subsurface  is a positively braided surface bounded by a two-component unlink. By Lemma~\ref{lem:unlink}, this subsurface is isotopic rel boundary to the standard Seifert surface for the two-component unlink. It follows that $\beta$ determines the same slice disk as $b$.

On the other hand, it is not as clear that $D'$ (which is determined by $b'$)  can be realized as a positively braided surface. We can view $D'$ as being obtained by attaching the band $b'$ (corresponding to the first occurrence of the generator $\sigma_2$ in $\beta'$) to the standard pair of Seifert disks bounded by the braided two-component unlink
\begin{equation}\label{eqn:braid}
\sigma_{3}\sigma_{2}^{-1}(\sigma_{4}^{-1}\sigma_{3}^{-1}\sigma_{1}^{-2}\sigma_{2}\sigma_{3}\sigma_{2}^{-1}\sigma_{1}^2\sigma_{3}\sigma_{4})(\sigma_{1}^{-1}\sigma_{2}\sigma_{3}\sigma_{2}^{-1}\sigma_{1})(\sigma_{3}\sigma_{4}\sigma_{3}^{-1}) \in B_5.
\end{equation}
The proof of Lemma~\ref{lem:braid} in the appendix shows that the non-quasipositive braidword in \eqref{eqn:braid} and the quasipositive braidword
$$(w \sigma_{2}^{-1}\sigma_{1}\sigma_{2}w^{-1})(w\sigma_{2}^{-1}\sigma_{3}\sigma_{1}\sigma_{2}\sigma_{1}^{-1}\sigma_{3}^{-1}\sigma_{2}w^{-1}) (w\sigma_{3}^2\sigma_{4}\sigma_{3}^{-2}w^{-1}) \subset \beta'$$
are equivalent in the braid group $B_5$. Therefore the resulting braided unlink bounds a positively braided surface in the bidisk. By Lemma~\ref{lem:unlink}, any such positively braided surface will be isotopic (rel boundary) to a standard pair of Seifert disks, allowing us to realize $D'$ itself as a positively braided surface.
\end{proof}

\begin{proof}[Proof of Theorem~\ref{thm:factor}] The preceding proof shows that there is a braid isotopy from $\beta$ to $\beta'$, and this extends to a topological isotopy between the associated braided surfaces using Proposition~\ref{prop:top-iso}.  For the sake of contradiction, suppose that the braid isotopy from $\beta$ to $\beta'$ extends to a smooth isotopy between the associated braided surfaces. Attaching the band $c$ (i.e., adding $\sigma_2$ to the beginning of each braid), we obtain an isotopy between the annuli $A$ and $A'$. However, the pairs $(B^4,A)$ and $(B^4,A')$ are not diffeomorphic (by Propositions~\ref{prop:spheres} and \ref{prop:no-spheres}), yielding the desired contradiction.
\end{proof}

\begin{figure}\center
\def\svgwidth{.85\linewidth}
\begingroup%
  \makeatletter%
  \providecommand\color[2][]{%
    \errmessage{(Inkscape) Color is used for the text in Inkscape, but the package 'color.sty' is not loaded}%
    \renewcommand\color[2][]{}%
  }%
  \providecommand\transparent[1]{%
    \errmessage{(Inkscape) Transparency is used (non-zero) for the text in Inkscape, but the package 'transparent.sty' is not loaded}%
    \renewcommand\transparent[1]{}%
  }%
  \providecommand\rotatebox[2]{#2}%
  \newcommand*\fsize{\dimexpr\f@size pt\relax}%
  \newcommand*\lineheight[1]{\fontsize{\fsize}{#1\fsize}\selectfont}%
  \ifx\svgwidth\undefined%
    \setlength{\unitlength}{2039.18637734bp}%
    \ifx\svgscale\undefined%
      \relax%
    \else%
      \setlength{\unitlength}{\unitlength * \real{\svgscale}}%
    \fi%
  \else%
    \setlength{\unitlength}{\svgwidth}%
  \fi%
  \global\let\svgwidth\undefined%
  \global\let\svgscale\undefined%
  \makeatother%
  \begin{picture}(1,0.23830243)%
    \lineheight{1}%
    \setlength\tabcolsep{0pt}%
    \put(0,0){\includegraphics[width=\unitlength,page=1]{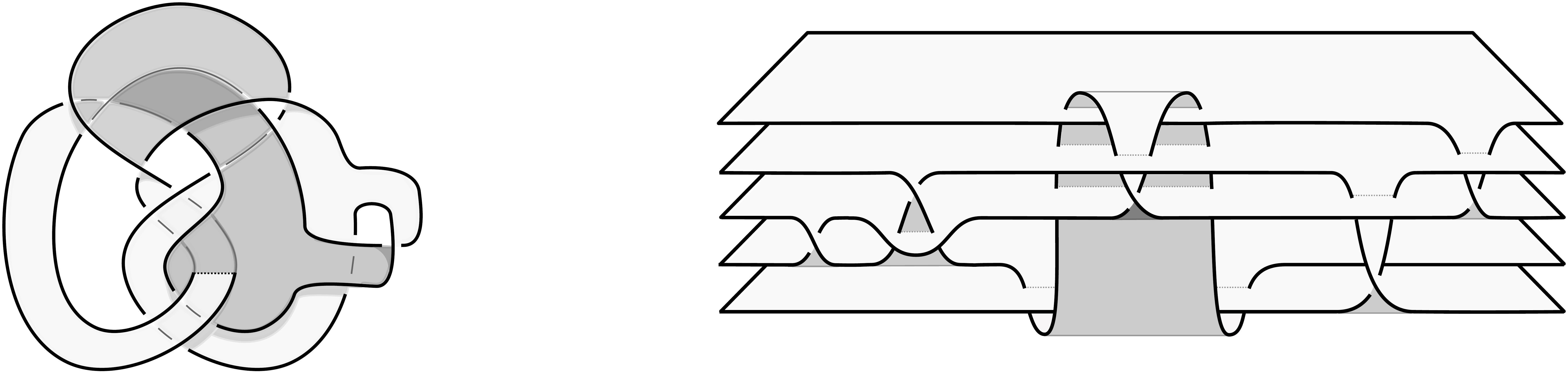}}%
    \put(0.342159,0.11470614){\makebox(0,0)[lt]{\lineheight{1.25}\smash{\begin{tabular}[t]{l}$=$\end{tabular}}}}%
  \end{picture}%
\endgroup%

\caption{The annulus $A$ as a positively braided surface.}\label{fig:A-and-T}
\end{figure}

\begin{figure}[b]\center
\def\svgwidth{.8\linewidth}
\begingroup%
  \makeatletter%
  \providecommand\color[2][]{%
    \errmessage{(Inkscape) Color is used for the text in Inkscape, but the package 'color.sty' is not loaded}%
    \renewcommand\color[2][]{}%
  }%
  \providecommand\transparent[1]{%
    \errmessage{(Inkscape) Transparency is used (non-zero) for the text in Inkscape, but the package 'transparent.sty' is not loaded}%
    \renewcommand\transparent[1]{}%
  }%
  \providecommand\rotatebox[2]{#2}%
  \newcommand*\fsize{\dimexpr\f@size pt\relax}%
  \newcommand*\lineheight[1]{\fontsize{\fsize}{#1\fsize}\selectfont}%
  \ifx\svgwidth\undefined%
    \setlength{\unitlength}{1776.50569405bp}%
    \ifx\svgscale\undefined%
      \relax%
    \else%
      \setlength{\unitlength}{\unitlength * \real{\svgscale}}%
    \fi%
  \else%
    \setlength{\unitlength}{\svgwidth}%
  \fi%
  \global\let\svgwidth\undefined%
  \global\let\svgscale\undefined%
  \makeatother%
  \begin{picture}(1,0.30637726)%
    \lineheight{1}%
    \setlength\tabcolsep{0pt}%
    \put(0,0){\includegraphics[width=\unitlength,page=1]{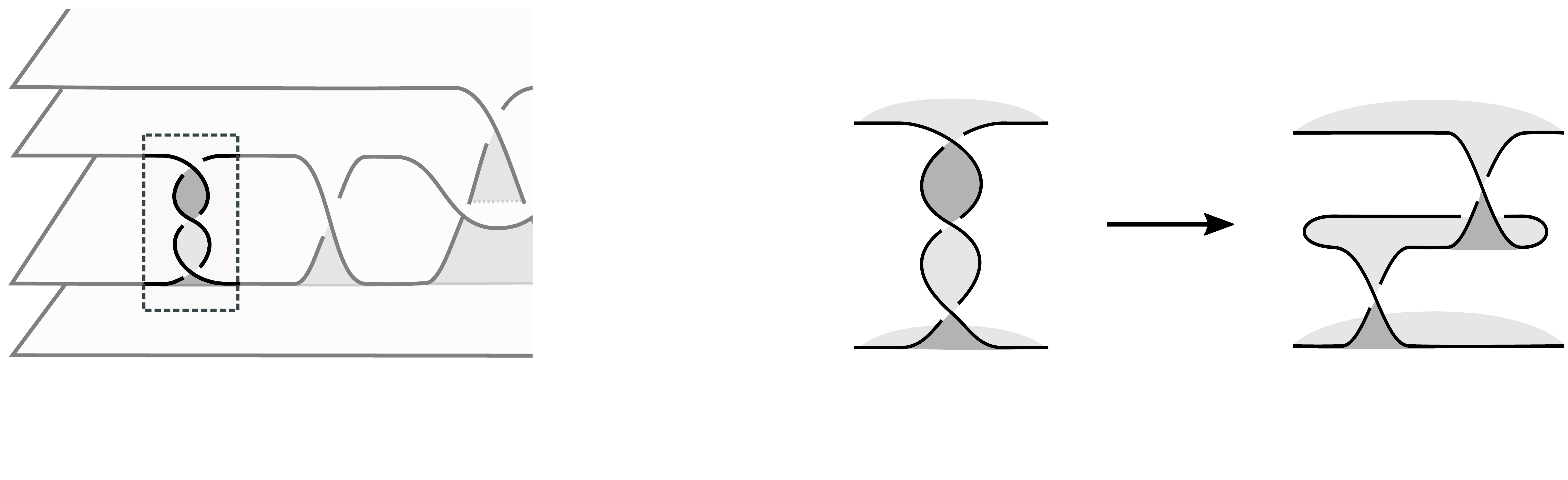}}%
    \put(0.14699207,0.00829697){\makebox(0,0)[lt]{\lineheight{1.25}\smash{\begin{tabular}[t]{l}(a)\end{tabular}}}}%
    \put(0.73011948,0.00829697){\makebox(0,0)[lt]{\lineheight{1.25}\smash{\begin{tabular}[t]{l}(b)\end{tabular}}}}%
  \end{picture}%
\endgroup%

\caption{Adding and manipulating a band with multiple positive twists.}\label{fig:twisted-band}
\end{figure}

Next, we recast the surfaces from Theorem~\ref{thm:ball} as positively braided surfaces.

\begin{thm}\label{thm:braided}
Fix the topological type $\mathcal{F}$ of any compact, connected, orientable surface with boundary other than the disk. There exist pairs of positively braided surfaces $F,F' \subset D^2 \! \times \! D^2$ of topological type $\mathcal{F}$ that are isotopic through homeomorphisms of $D^2 \! \times \! D^2$ (rel boundary) yet are not equivalent under diffeomorphisms of $D^2 \! \times \! D^2$.

Moreover, these surfaces may be chosen so that the double branched cover of $D^2 \! \times \! D^2$ over $F'$ contains a smoothly embedded 2-sphere of self-intersection $-2$, whereas the double branched cover of $F$ does not.
\end{thm}

\begin{proof}
It suffices to show that the surfaces $F,F' \subset B^4$ with the topological type of $\mathcal{F}$ constructed in the proof of Theorem~\ref{thm:ball} can be realized as positively braided surfaces with the same boundary. We begin by considering  the annuli $A$ and $A'$. Each of these is obtained from a positively braided surface (namely $D$ or $D'$) by attaching a positively braided band (namely $\sigma_2$), hence are themselves positively braided; see Figure~\ref{fig:A-and-T} for a depiction of $A$ as a braided surface. 

For the tori $T$ and $T'$, we attach positively twisted bands to $A$ and $A'$, but these bands are \emph{not} braided. However, as illustrated in Figure~\ref{fig:twisted-band}, such a band can be realized by introducing an additional disk and two braided bands into the braided surface; here the additional disk lies between the two disks that are joined by the band in question, and this new disk passes behind all existing bands. It follows that $T$ and $T'$ are smoothly isotopic (rel boundary) to positively braided surfaces. 

There are several ways to show that $A_0$ and $T_0$ are smoothly isotopic to positively braided surfaces. For example, one may begin with the  positively braided disk $D_0$ associated to the braid $\sigma_1 \in B_2$ (built from two disks and a half-twisted band). To obtain $A_0$, attach a band with three positive half-twists  to $D_0$ on the left side of the original $\sigma_1$ band (similarly to Figure~\ref{fig:twisted-band}).  To obtain $T_0$, attach a second multiply-twisted band of the same form to the \emph{right} of the original $\sigma_1$ band. 

All of the remaining surfaces in the proof of Theorem~\ref{thm:ball} are obtained as boundary sums of the preceding building blocks. It is straightforward to see that boundary sums of positively braided surfaces can be realized as positively braided surfaces in the same way as a typical connected sum of braids: the two surfaces are stacked vertically and then joined by a positively half-twisted band attached between the bottommost disk of the upper surface and the topmost disk of the lower surface.
\end{proof}

Finally, we use the descriptions of $D$, $A$, and $T$ as braided surfaces to produce handle diagrams for the double branched covers $\Sigma(B^4,D)$, $\Sigma(B^4,A)$, and $\Sigma(B^4,T)$.

\begin{proof}[Proof of Lemma~\ref{lem:DBCs}]
This handle calculus is similar in spirit to that of Figures~\ref{fig:Dn-exterior}-\ref{fig:dbc-family}, so we relegate these additional (and rather large) figures to \S\ref{sec:handles} in the appendix. A handle diagram for the exterior of $A$ in $B^4$ is drawn and simplified in Figure~\ref{fig:exterior}. The handle diagram for the double branched cover $\Sigma(B^4,A)$ is then produced in Figure~\ref{fig:dbc}. The handle diagram for $\Sigma(B^4,T)$ is obtained similarly; the necessary modifications are sketched in Figure~\ref{fig:torus-dbc}.
\end{proof}

\vspace{-.35in}

\section{Exotically knotted holomorphic curves in $\boldsymbol{\mathbf{C}^2}$}\label{sec:plane}

A \emph{Fatou-Bieberbach domain} is a proper open subset $\Omega \subsetneq \cc^2$ that is biholomorphically equivalent to $\cc^2$. The use of these subsets to produce proper embeddings of complex curves into $\cc^2$ dates back to  Kasahara and Nishino (cf \cite{yanagihara,stehle}), who constructed the first proper holomorphic embeddings of the open unit disk  $\mathring{D}^2 \subset \cc$ into $\cc^2$.  They  located a Fatou-Bieberbach domain $\Omega \subset \cc^2$ and a complex curve $V \subset \cc^2$ such that one of the connected components of the intersection $V \cap \Omega$ is an open disk; this furnishes the desired embedding  of the open disk into $\cc^2$ via the biholomorphism $\Omega \cong \cc^2$.

While this is a powerful constructive technique, it reveals little about the topology of the resulting embedding. Closer to our needs is a beautiful incarnation of this strategy due to Baader, Kutzschebauch, and Wold \cite{bkw}, who showed that  proper holomorphic embeddings of the open disk into $\cc^2$ can be \emph{topologically knotted}, i.e., not  equivalent to $\cc \! \times \! 0$. Their disks' complements have $\pi_1 \not \cong \zz$, distinguishing them from  $\cc \! \times \! 0$.  
 Our strategy is similar, though care must be taken to preserve both the topological equivalence between our curves and the smooth obstruction distinguishing them.

As in \cite{bkw}, we will use the following special case of \cite[Theorem 1.1]{globevnik} to help us convert  compact pieces of algebraic curves in the bidisk into properly embedded holomorphic curves in $\cc^2$. We use $D^2_R$ to denote the disk of radius $R$ in $\cc$.

\begin{thm}[Globevnik]\label{thm:fb}
For any $R>1$, there exists a domain $\Omega \subset \cc^2$ and a biholomorphic map from $\Omega$ onto $\cc^2$ such that:
\begin{enumerate}[label=\normalfont(\roman*)]
\item $\Omega$ is contained in $\left\{ (z,w) \in \cc^2 : |z| < \max(R,|w|) \right\}$, and 

\item $\Omega \cap (\mathring{D}^2_R \! \times \!  \mathring{D}^2_R)$ is an arbitrarily small $C^1$-perturbation of $\mathring{D}^2_1 \! \times  \! \mathring{D}^2_R$.

\end{enumerate}
\end{thm}

Globevnik constructs  $\Omega$ so that its frontier $\overline{\Omega} \setminus \Omega$  is $C^1$-differentiable and  intersects $\mathring{D}^2_R \! \times\!  \mathring{D}^2_R$ in a $C^1$-small perturbation of $\partial {D}^2_1 \! \times \! \mathring{D}^2_R$. 

\begin{rem} Below, we will often be fixing a surface in $\cc^2$ and considering  its intersections   with a variety of domains and subspaces of $\cc^2$. For convenience, we take the following notational convention: Given subspaces $X$ and $Y$ of a topological space $Z$, we write $(X,Y)$ to mean the space-subspace pair $(X, Y \cap X)$.
\end{rem}

\begin{proof}[Proof of Theorem~\ref{thm:complex}]
For notational convenience, we will write $\dd$ for the closed unit bidisk $D^2_1 \! \times \! D^2_1$. By Theorem~\ref{thm:braided} and Theorem~\ref{thm:rudolph}, there are  smooth algebraic curves $V_0,V_1 \subset \cc^2$ whose intersections with the bidisk are positively braided surfaces $F_0= V_0 \cap \dd$ and $F_1 = V_1 \cap \dd$ of the desired topological type that satisfy the following:
\begin{enumerate}
\item [(i)] $F_0$ and $F_1$ are isotopic through homeomorphisms of $\dd$ that are smooth near $\partial \dd$, and each intermediate surface $F_t$ has braided boundary  in $\partial D^2 \! \times \! D^2 \subset \dd$;
\item [(ii)] the double branched cover of $(\dd,F_0)$ contains a smoothly embedded 2-sphere of self-intersection $-2$; and
\item [(iii)]  the double branched cover of $(\dd,F_1)$ contains no smoothly embedded 2-spheres of self-intersection $-2$.
\end{enumerate}

For convenience, we may smoothly extend each $F_t$ slightly beyond the unit bidisk $\dd$ to a larger bidisk $\dd_{+\epsilon}$; see the left side of Figure~\ref{fig:schematic}. Note that, since $F_0$ and $F_1$ are compact pieces of larger algebraic curves in $\cc^2$, we may choose their extensions to remain holomorphic. Moreover, if $\epsilon>0$ is chosen sufficiently small, then for each fixed $t \in [0,1]$ we may assume that the pairs $(\dd, F_t)$, $(\dd_{+\epsilon},F_t)$, and $(\dd_{-\epsilon},F_t)$ are diffeomorphic. (Of course, $F_t$ is not smooth for all $t \in (0,1)$, but nothing precludes us from comparing non-smooth surfaces under diffeomorphisms of their ambient spaces.)

\begin{figure}\center
\smallskip
\def\svgwidth{\linewidth}
\begingroup%
  \makeatletter%
  \providecommand\color[2][]{%
    \errmessage{(Inkscape) Color is used for the text in Inkscape, but the package 'color.sty' is not loaded}%
    \renewcommand\color[2][]{}%
  }%
  \providecommand\transparent[1]{%
    \errmessage{(Inkscape) Transparency is used (non-zero) for the text in Inkscape, but the package 'transparent.sty' is not loaded}%
    \renewcommand\transparent[1]{}%
  }%
  \providecommand\rotatebox[2]{#2}%
  \newcommand*\fsize{\dimexpr\f@size pt\relax}%
  \newcommand*\lineheight[1]{\fontsize{\fsize}{#1\fsize}\selectfont}%
  \ifx\svgwidth\undefined%
    \setlength{\unitlength}{2621.10830232bp}%
    \ifx\svgscale\undefined%
      \relax%
    \else%
      \setlength{\unitlength}{\unitlength * \real{\svgscale}}%
    \fi%
  \else%
    \setlength{\unitlength}{\svgwidth}%
  \fi%
  \global\let\svgwidth\undefined%
  \global\let\svgscale\undefined%
  \makeatother%
  \begin{picture}(1,0.23549089)%
    \lineheight{1}%
    \setlength\tabcolsep{0pt}%
    \put(0,0){\includegraphics[width=\unitlength,page=1]{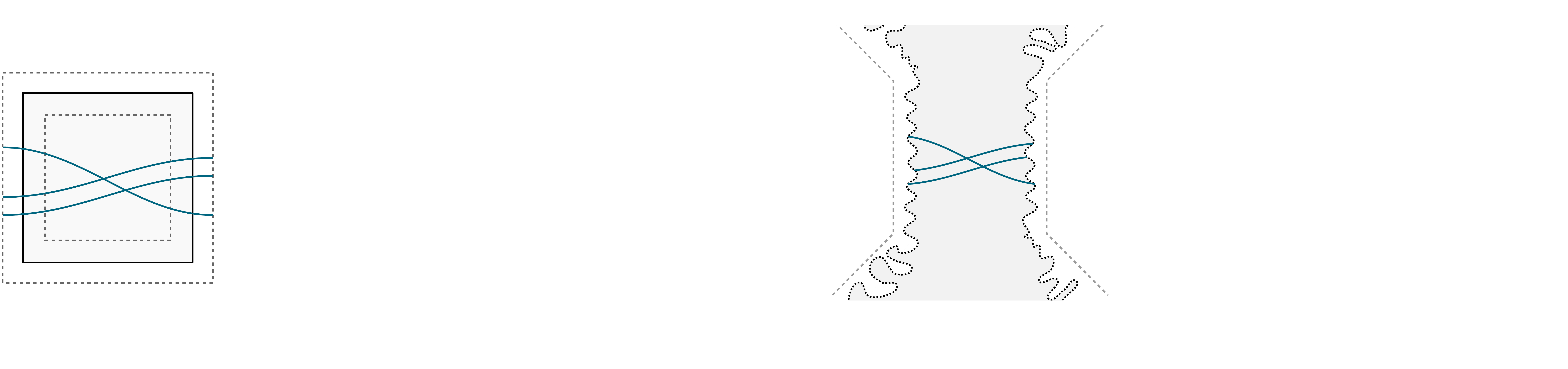}}%
    \put(0.02548226,-0.00050565){\makebox(0,0)[lt]{\lineheight{1.25}\smash{\begin{tabular}[t]{l}$F_t \subset \mathbb{D}$\end{tabular}}}}%
    \put(0,0){\includegraphics[width=\unitlength,page=2]{schematic.pdf}}%
    \put(0.98146962,0.14107106){\makebox(0,0)[lt]{\lineheight{1.25}\smash{\begin{tabular}[t]{l}$U$\end{tabular}}}}%
    \put(0,0){\includegraphics[width=\unitlength,page=3]{schematic.pdf}}%
    \put(0.3182817,-0.00059232){\makebox(0,0)[lt]{\lineheight{1.25}\smash{\begin{tabular}[t]{l}$W$\end{tabular}}}}%
    \put(0.82947434,-0.00041905){\color[rgb]{0,0,0}\makebox(0,0)[lt]{\lineheight{1.25}\smash{\begin{tabular}[t]{l}$\partial \bar \Omega=\bar \Omega \setminus \Omega$\end{tabular}}}}%
    \put(0.60966818,-0.00041905){\color[rgb]{0,0,0}\makebox(0,0)[lt]{\lineheight{1.25}\smash{\begin{tabular}[t]{l}$\Omega$\end{tabular}}}}%
    \put(0.14793276,0.17061813){\color[rgb]{0.6,0.6,0.6}\makebox(0,0)[lt]{\lineheight{1.25}\smash{\begin{tabular}[t]{l}$\mathbb{D}_{+\epsilon}$\end{tabular}}}}%
  \end{picture}%
\endgroup%

\caption{A schematic describing of the relationship between the surfaces $F_t$, the bidisk $\dd$, and the Fatou-Bieberbach domain $\Omega$.}\label{fig:schematic}
\end{figure}

We wish to find a Fatou-Bieberbach domain $\Omega \subset \cc^2$ that lies within the set
$$W =\left \{(z,w) \in \cc^2 : |z| \leq \max(1+\epsilon,|w|)\right\},$$
as depicted in the middle two diagrams of Figure~\ref{fig:schematic}. We will further demand that, near the unit bidisk $\dd$, the ``boundary'' of $\Omega$ (technically, its \emph{frontier}   $\partial \bar{\Omega} = \bar{\Omega} \setminus \Omega$) is sandwiched between $\partial \dd_{-\epsilon}$ and $\partial \dd_{+\epsilon}$, as depicted on the right side of Figure~\ref{fig:schematic}; this will provide control over the restricted surfaces $F_t \cap \Omega$. 

To that end,   fix a constant $R>1+\epsilon$. By Theorem~\ref{thm:fb}, we can find a Fatou-Bieberbach domain $\Omega \subset \cc^2$ contained in $\left\{(z,w) \in \cc^2 : |z| < \max(R,|w|)\right\}$ such that the subset
$$ \Omega_{|w|\leq R}=\Omega \cap \{(z,w) \in \cc^2 : |w| \leq R\}$$
is an arbitrarily small $C^1$-perturbation of $\mathring{D}_1^2 \! \times \! \mathring{D}^2_R$. Near the bidisk $\dd$, the frontier $\partial \bar{\Omega} = \bar{\Omega} \setminus \Omega$ consists of a $C^1$-regular solid torus
$$ U = \partial \bar{\Omega} \cap  \{(z,w) \in \cc^2 : |w| \leq 1\}$$ obtained as a small $C^1$-perturbation of the smooth solid torus $\partial D^2_1 \! \times \! D_1^2 \subset \dd$.  
By constructing $\Omega$ so that this perturbation is  sufficiently small, we may assume that $U$ lies in $\dd_{+\epsilon} \setminus \dd_{-\epsilon}$, as desired. As a final condition, we may further assume that  the $C^1$-perturbation is small enough to ensure that $F_t$  is transverse to the $C^1$-regular solid torus $U$ for all $t \in [0,1]$. This is possible because $F_t$ (which is  smooth near $|z|=1$) is  transverse to $\partial D^2_1 \! \times \! D^2_1 \subset \dd$ for all $t$, and $U$ is an arbitrarily small $C^1$-perturbation of $\partial D^2_1 \! \times \! D^2_1$. This final assumption on $\Omega$ ensures that the surface $F_t \cap \bar \Omega$ is properly embedded in $\bar \Omega$ for all $t \in [0,1]$, hence defines a family of isotopic surfaces in $\bar \Omega$.

We next show that the isotopy of surfaces $F_t \cap \Omega$ is induced by an ambient isotopy through homeomorphisms of $\Omega$. (Note that the ambient isotopy of $\dd$ --- or $\dd_{+\epsilon}$, technically --- need not carry $\Omega$ to itself, for example.) 
 We wish to apply the isotopy extension theorem to the family $F_t \cap \Omega$ produce an ambient isotopy through homeomorphisms of $\Omega$. However, $F_t \cap \Omega$ is not compact. To sidestep this, we  consider the compact surfaces $F_t \cap \bar \Omega$. Unfortunately, the closure $\bar \Omega$ is not guaranteed to be a manifold. Fortunately, each surface $F_t \cap \bar \Omega$ lies in the subset $$\bar \Omega_{|w| \leq 1} =\bar \Omega \cap \{(z,w) \in \cc^2 : |w| \leq 1\},$$ which is indeed a $C^1$-manifold (obtained as a $C^1$-perturbation of $\dd$). We may then apply the isotopy extension theorem  (e.g., \cite[Corollary~1.2]{edwards-kirby}) to extend the isotopy of surfaces $F_t \cap \bar \Omega$ to an isotopy through ambient homeomorphisms of $\bar \Omega_{|w| \leq 1}$. Note that these ambient homeomorphisms must carry the interior of  $\bar \Omega_{|w| \leq 1}$ to itself. Since $F_t$ lies inside the region $\{(z,w) \in \cc^2 : |w| < 1\}$, we may further assume that the isotopy fixes all points with $|w| \geq 1$. It follows that this isotopy of $\bar \Omega_{|w| \leq 1}$ can be extended by the identity to the rest of $\bar \Omega$. Restricting again to the interior, we obtain the desired isotopy of the pairs $(\Omega,F_t)$ through ambient homeomorphisms of $\Omega$.

Since $\Omega$ is biholomorphically equivalent to $\cc^2$, the surfaces $F_t \cap \Omega$ define a family of properly embedded surfaces in $\cc^2$ that are isotopic through ambient homeomorphisms. Moreover, these surfaces are embedded as holomorphic curves for $t=0$ and $t=1$. To show that these holomorphic curves in $\cc^2$ are not smoothly isotopic, it suffices to distinguish $(\Omega, F_0)$ and $(\Omega, F_1)$. To that end, we recall that $\Omega$ contains the shrunken bidisk $\dd_{-\epsilon}$. This implies that the double branched cover of $(\Omega,F_0)$ contains the double branched cover of $(\dd_{-\epsilon},F_0)$. By our choice of $\epsilon$ from above, the pairs $(\dd_{-\epsilon},F_0)$ and $(\dd,F_0)$ are diffeomorphic. By (ii) above, the double branched cover of $(\dd,F_0)$ contains a smoothly embedded 2-sphere of self-intersection $-2$. It  follows that the double branched cover of $(\Omega,F_0)$ contains a smooth 2-sphere of self-intersection $-2$.

It remains to show that the double branched cover of $(\Omega,F_1)$ contains no smooth 2-spheres of self-intersection $-2$. To this end, we recall that $\Omega$ lies inside  $$W = \left \{(z,w) \in \cc^2 :  |z| <  \max(1+\epsilon,|w|) \right\}.$$
Therefore the pair $(\Omega,F_1)$ includes into the pair $(W,F_1)$ and the double branched cover of $(\Omega,F_1)$ embeds smoothly in the double branched cover of $(W,F_1)$. It is straightforward to construct a diffeomorphism of pairs $(W,F_1) \cong (\dd_{+\epsilon},F_1)$ that is the identity on $\{|w| \leq 1\}$. From above, we have $(\dd_{+\epsilon},F_1) \cong (\dd,F_1)$. By (iii), the double branched cover of $(\dd,F_1)$ contains no smoothly embedded 2-spheres of self-intersection $-2$, so the same must be true of the double branched cover of $(W,F_1)$. This in turn implies that the same is true of the double branched cover of $(\Omega,F_1)$, as claimed. 
\end{proof}

\titleformat{\section}{\large\bfseries}{}{0pt}{\center }


\titlespacing{\section}{0pt}{*4}{*1.5}
\setcounter{section}{0}
\renewcommand{\thesection}{\Alph{section}}\setcounter{subsection}{1}
\titleformat{\subsection}[runin]{\bfseries}{}{0pt}{\thesection.\arabic{subsection} \ \ }

\vspace{-.35in}

\section{Appendix}

\subsection{Braid words.}\label{app:braid} We now verify a calculation from the proof of Theorem~\ref{thm:braid}.

\begin{lem}\label{lem:braid}
The following elements are equal in the 5-stranded braid group:
\begin{align*}
\beta&=(\sigma_{2}\sigma_{3}\sigma_{2}^{-1})(\sigma_1^{-2} \sigma_2 \sigma_3 \sigma_4^2 \sigma_3^{-1} \sigma_2 \sigma_3 \sigma_4^{-2}  \sigma_3^{-1} \sigma_2^{-1} \sigma_1^2)(\sigma_3^{-1} \sigma_2 \sigma_1 \sigma_2^{-1} \sigma_3)(\sigma_4^{-1} \sigma_3 \sigma_4)\\
\beta'&=\sigma_{2}(w \sigma_{2}^{-1}\sigma_{1}\sigma_{2}w^{-1})(w\sigma_{2}^{-1}\sigma_{3}\sigma_{1}\sigma_{2}\sigma_{1}^{-1}\sigma_{3}^{-1}\sigma_{2}w^{-1}) (w\sigma_{3}^2\sigma_{4}\sigma_{3}^{-2}w^{-1}),
\end{align*}
where $w = \sigma_3 \sigma_4^{-1} \sigma_1^{-1} \sigma_3^{-2} \sigma_2^{-1}   \sigma_1^{-1} \sigma_3^{-1}$.
\end{lem}
\begin{proof}
 In each step, we underline the symbols that are to be modified by a relation in the braid group (especially the useful relation $\sigma_i \sigma_{i+1} \sigma_i^{-1}= \sigma_{i+1}^{-1} \sigma_i \sigma_{i+1}$).
\smallskip
\begin{align*}
\beta &= ( \sigma_2  \sigma_3  \sigma_2^{-1}  )( \sigma_1^{-1}   \sigma_1^{-1}   \sigma_2  \sigma_3  \sigma_4  \sigma_4  \sigma_3^{-1}   \sigma_2  \sigma_3  \sigma_4^{-1}   \sigma_4^{-1}   \sigma_3^{-1}   \sigma_2^{-1}   \sigma_1  \sigma_1 )( \sigma_3^{-1}   \underline{\sigma_2  \sigma_1  \sigma_2^{-1}}  \sigma_3 )( \sigma_4^{-1}   \sigma_3  \sigma_4 )
\\
&=
( \sigma_2  \sigma_3  \sigma_2^{-1}  )( \sigma_1^{-1}   \sigma_1^{-1}   \sigma_2  \sigma_3  \sigma_4  \sigma_4  \sigma_3^{-1}   \sigma_2  \sigma_3  \sigma_4^{-1}   \sigma_4^{-1}   \sigma_3^{-1}   \sigma_2^{-1}   \sigma_1  \sigma_1 )(\underline{\sigma_3^{-1}   \sigma_1^{-1}}  \sigma_2  \underline{\sigma_1   \sigma_3)( \sigma_4^{-1}   \sigma_3  \sigma_4 )}
\\
&=
( \sigma_2  \sigma_3  \sigma_2^{-1}  )( \sigma_1^{-1}   \sigma_1^{-1}   \sigma_2  \sigma_3  \sigma_4  \sigma_4  \sigma_3^{-1}   \sigma_2  \sigma_3  \sigma_4^{-1}   \sigma_4^{-1}   \sigma_3^{-1}   \sigma_2^{-1}   \sigma_1  \sigma_1 )( \sigma_1^{-1}   \sigma_3^{-1}  \sigma_2  \sigma_3   )( \sigma_4^{-1}   \sigma_3  \sigma_4 )\sigma_1 
\\
&=
( \sigma_2  \sigma_3  \sigma_2^{-1}  )( \sigma_1^{-1}   \sigma_1^{-1}   \sigma_2  \sigma_3  \sigma_4  \sigma_4  \sigma_3^{-1}   \sigma_2  \sigma_3  \sigma_4^{-1}   \sigma_4^{-1}   \sigma_3^{-1}   \sigma_2^{-1}  \sigma_1  \underline{\sigma_1 )( \sigma_1^{-1}}   \sigma_3^{-1}  \sigma_2  \sigma_3  )( \sigma_4^{-1}   \sigma_3  \sigma_4 )\sigma_1 
\\
&=
( \sigma_2  \sigma_3  \sigma_2^{-1}  )( \sigma_1^{-1}   \sigma_1^{-1}   \sigma_2  \sigma_3  \sigma_4  \sigma_4  \sigma_3^{-1}   \sigma_2  \sigma_3  \sigma_4^{-1}   \sigma_4^{-1}   \sigma_3^{-1}   \sigma_2^{-1}  \sigma_1   )(  \sigma_3^{-1}  \sigma_2  \sigma_3   )( \underline{\sigma_4^{-1}   \sigma_3  \sigma_4} )\sigma_1
\\
&=
( \sigma_2  \sigma_3  \sigma_2^{-1}  )( \sigma_1^{-1}   \sigma_1^{-1}   \sigma_2  \sigma_3  \sigma_4  \sigma_4  \sigma_3^{-1}   \sigma_2  \sigma_3  \sigma_4^{-1}   \sigma_4^{-1}   \sigma_3^{-1}   \sigma_2^{-1}  \sigma_1   )( \underline{\sigma_3^{-1}  \sigma_2  \sigma_3 })( \sigma_3   \sigma_4  \sigma_3^{-1} )\sigma_1
\\
&=
( \sigma_2  \sigma_3  \sigma_2^{-1}  )( \sigma_1^{-1}   \sigma_1^{-1}   \sigma_2 \underline{\sigma_3  \sigma_4  \sigma_4  \sigma_3^{-1}}  \sigma_2  \underline{\sigma_3  \sigma_4^{-1}   \sigma_4^{-1}   \sigma_3^{-1}}   \sigma_2^{-1}  \sigma_1   )(  \sigma_2  \sigma_3  \sigma_2^{-1}   )( \sigma_3   \sigma_4  \sigma_3^{-1} )\sigma_1
\\
&=
( \sigma_2  \sigma_3  \sigma_2^{-1}  )( \sigma_1^{-1}   \sigma_1^{-1}   \sigma_2 \sigma_4^{-1}  \sigma_3  \sigma_3  \underline{\sigma_4  \sigma_2  \sigma_4^{-1}}  \sigma_3^{-1}   \sigma_3^{-1}   \sigma_4   \sigma_2^{-1}  \sigma_1   )(  \sigma_2  \sigma_3  \sigma_2^{-1}   )( \sigma_3   \sigma_4  \sigma_3^{-1} )\sigma_1
\\
&=
( \sigma_2  \sigma_3  \sigma_2^{-1}  )( \sigma_1^{-1}   \sigma_1^{-1}   \sigma_2 \sigma_4^{-1}  \sigma_3  \sigma_3   \sigma_2 \sigma_3^{-1}   \sigma_3^{-1}   \underline{\sigma_4   \sigma_2^{-1}}  \sigma_1   )(  \sigma_2  \sigma_3  \sigma_2^{-1}   )( \sigma_3   \sigma_4  \sigma_3^{-1} )\sigma_1
\\
&=
( \sigma_2  \sigma_3  \sigma_2^{-1}  )( \sigma_1^{-1}   \sigma_1^{-1}   \sigma_2 \sigma_4^{-1}  \sigma_3  \sigma_3   \underline{\sigma_2 \sigma_3^{-1}   \sigma_3^{-1}      \sigma_2^{-1} } \sigma_4 \sigma_1   )(  \sigma_2  \sigma_3  \sigma_2^{-1}   )( \sigma_3   \sigma_4  \sigma_3^{-1} )\sigma_1
\\
&=
( \sigma_2  \sigma_3  \sigma_2^{-1}  )( \sigma_1^{-1}   \sigma_1^{-1}   \sigma_2 \sigma_4^{-1}  \sigma_3  \underline{\sigma_3   \sigma_3^{-1} } \sigma_2^{-1}   \sigma_2^{-1}      \sigma_3  \sigma_4 \sigma_1   )(  \sigma_2  \sigma_3  \sigma_2^{-1}   )( \sigma_3   \sigma_4  \sigma_3^{-1} )\sigma_1
\\
&=
( \sigma_2  \sigma_3  \sigma_2^{-1}  )( \sigma_1^{-1}   \sigma_1^{-1}   \underline{\sigma_2 \sigma_4^{-1}}  \sigma_3   \sigma_2^{-1}   \sigma_2^{-1}      \sigma_3  \sigma_4 \sigma_1   )(  \sigma_2  \sigma_3  \sigma_2^{-1}   )( \sigma_3   \sigma_4  \sigma_3^{-1} )\sigma_1
\\
&=
( \sigma_2  \sigma_3  \sigma_2^{-1}  )( \sigma_1^{-1}   \sigma_1^{-1}    \sigma_4^{-1}  \underline{\sigma_2 \sigma_3   \sigma_2^{-1}}   \sigma_2^{-1}      \sigma_3  \sigma_4 \sigma_1   )(  \sigma_2  \sigma_3  \sigma_2^{-1}   )( \sigma_3   \sigma_4  \sigma_3^{-1} )\sigma_1
\\
&=
( \sigma_2  \sigma_3  \sigma_2^{-1}  ) \sigma_1^{-1}   (\underline{\sigma_1^{-1}    \sigma_4^{-1}  \sigma_3^{-1}} \sigma_2   \sigma_3   \sigma_2^{-1} \underline{\sigma_3  \sigma_4 \sigma_1}   )(  \sigma_2  \sigma_3  \sigma_2^{-1}   )( \sigma_3   \sigma_4  \sigma_3^{-1} )\sigma_1
\\
&=\hl{\sigma_{2}}\sigma_{3}\sigma_{2}^{-1}\sigma_{1}^{-1}(\sigma_{4}^{-1}\sigma_{3}^{-1}\sigma_{1}^{-1}\sigma_{2}\sigma_{3}\sigma_{2}^{-1}\sigma_{1}\sigma_{3}\sigma_{4}) (\sigma_{2}\sigma_{3}\sigma_{2}^{-1})\sigma_{3}\underline{\sigma_{4}\sigma_{3}^{-1}\sigma_{1}}
\\
&=\hl{\sigma_{2}}\sigma_{3} \underline{\sigma_{2}^{-1}\sigma_{1}^{-1}\sigma_{4}^{-1}}(\sigma_{3}^{-1}\sigma_{1}^{-1}\sigma_{2}\sigma_{3}\sigma_{2}^{-1}\sigma_{1}\sigma_{3})\sigma_{4}(\sigma_{2}\sigma_{3}\sigma_{2}^{-1})\sigma_{3}\sigma_{1}\sigma_{4}\sigma_{3}^{-1}
\\
&=\hl{\sigma_{2}}\underline{\sigma_{3}\sigma_{4}}^{-1}\sigma_{2}^{-1}\sigma_{1}^{-1}(\sigma_{3}^{-1}\sigma_{1}^{-1}\sigma_{2}\sigma_{3}\sigma_{2}^{-1}\sigma_{1}\sigma_{3})\sigma_{4}(\sigma_{2}\sigma_{3}\sigma_{2}^{-1})\sigma_{3}\sigma_{1} \underline{\sigma_{4}\sigma_{3}^{-1}}
\\
&=\hl{\sigma_{2}}(\sigma_{3}\sigma_{4}^{-1})\sigma_{2}^{-1}\sigma_{1}^{-1}(\sigma_{3}^{-1}\sigma_{1}^{-1} \underline{\sigma_{2}\sigma_{3}\sigma_{2}^{-1}} \sigma_{1}\sigma_{3})\sigma_{4}(\sigma_{2}\sigma_{3}\sigma_{2}^{-1})\sigma_{3}\sigma_{1}(\sigma_{4}\sigma_{3}^{-1})
\\
&=\hl{\sigma_{2}}(\sigma_{3}\sigma_{4}^{-1})\sigma_{2}^{-1}\sigma_{1}^{-1}(\sigma_{3}^{-1} \underline{\sigma_{1}^{-1}\sigma_{3}^{-1}} \sigma_{2} \underline{\sigma_{3}\sigma_{1}} \sigma_{3})\sigma_{4}(\sigma_{2}\sigma_{3}\sigma_{2}^{-1})\sigma_{3}\sigma_{1}(\sigma_{4}\sigma_{3}^{-1})
\\
&=\hl{\sigma_{2}}(\sigma_{3}\sigma_{4}^{-1})\sigma_{2}^{-1}\sigma_{1}^{-1}(\sigma_{3}^{-2}\sigma_{1}^{-1}\sigma_{2} \sigma_{1} \sigma_{3}^2)\sigma_{4}(\sigma_{2}\sigma_{3}\sigma_{2}^{-1})\sigma_{3}\sigma_{1}(\sigma_{4}\sigma_{3}^{-1})
\\
&=\hl{\sigma_{2}}(\sigma_{3}\sigma_{4}^{-1})\sigma_{2}^{-1}\sigma_{1}^{-1}\underline{\sigma_{3}^{-2} \sigma_{1}^{-1}} \sigma_{2} \underline{\sigma_{1} \sigma_{3}^2 \sigma_{4}} \sigma_{2}\sigma_{3}\sigma_{2}^{-1}\sigma_{3}\sigma_{1}(\sigma_{4}\sigma_{3}^{-1})
\\
&=\hl{\sigma_{2}}(\sigma_{3}\sigma_{4}^{-1})\sigma_{2}^{-1}\sigma_{1}^{-2} \sigma_{3}^{-2}  \sigma_{2} \sigma_{3}^2 \underline{\sigma_{4} \sigma_{1}}  \sigma_{2}\sigma_{3}\sigma_{2}^{-1}\sigma_{3}\sigma_{1}(\sigma_{4}\sigma_{3}^{-1})
\\
&=\hl{\sigma_{2}}(\sigma_{3}\sigma_{4}^{-1})\sigma_{2}^{-1}\sigma_{1}^{-2} \sigma_{3}^{-2}  \sigma_{2} \underline{\sigma_{3}^2 \sigma_{4} (\sigma_3^{-2}} \sigma_3^2 ) \sigma_{1}  \sigma_{2}\sigma_{3}\sigma_{2}^{-1}\sigma_{3}\sigma_{1}(\sigma_{4}\sigma_{3}^{-1})
\\
&=\hl{\sigma_{2}}(\sigma_{3}\sigma_{4}^{-1})\sigma_{2}^{-1}\sigma_{1}^{-2}\sigma_{3}^{-2}\sigma_{2}(\sigma_{3}^2 \sigma_{4}\sigma_{3}^{-2})\sigma_{3}^2\sigma_{1} \underline{\sigma_{2}\sigma_{3}\sigma_{2}^{-1}} \sigma_{3}\sigma_{1}(\sigma_{4}\sigma_{3}^{-1})
\\
&=\hl{\sigma_{2}}(\sigma_{3}\sigma_{4}^{-1})\sigma_{2}^{-1}\sigma_{1}^{-2}\sigma_{3}^{-2}\sigma_{2}(\sigma_{3}^2 \sigma_{4}\sigma_{3}^{-2}) \underline{\sigma_{3}^2 \sigma_{1}\sigma_{3} ^{-1}} \sigma_{2}\sigma_{3}\sigma_{3}\sigma_{1}(\sigma_{4}\sigma_{3}^{-1})
\\
&=\hl{\sigma_{2}}(\sigma_{3}\sigma_{4}^{-1})\sigma_{2}^{-1}\sigma_{1}^{-2}\sigma_{3}^{-2}\sigma_{2}(\sigma_{3}^2\sigma_{4}\sigma_{3}^{-2})\sigma_{3}\sigma_{1} \sigma_{2}\sigma_{3}^2\sigma_{1}(\sigma_{4}\sigma_{3}^{-1})
\end{align*}

Set $x=\sigma_{3}
\sigma_{1}\sigma_{2}\sigma_{3}^2\sigma_{1}$. Continuing from above, we have:
\begin{align*}
\cdots \ &=\hl{\sigma_{2}}(\sigma_{3}\underline{\sigma_{4}^{-1}) (\sigma_{2}}^{-1}\sigma_{1}^{-2}\sigma_{3}^{-2}\sigma_{2})(\sigma_{3}^2 \sigma_{4}\sigma_{3}^{-2})x(\sigma_{4}\sigma_{3}^{-1})
\\
&=\hl{\sigma_{2}}(\sigma_{3}\sigma_{4}^{-1})x^{-1} \underline{x}(\sigma_{2}^{-1}\sigma_{1}^{-2}\sigma_{3}^{-2}\sigma_{2})(\sigma_{3}^2 \sigma_{4}\sigma_{3}^{-2})x(\sigma_{4}\sigma_{3}^{-1})
\\
&=\hl{\sigma_{2}}(\sigma_{3}\sigma_{4}^{-1})x^{-1}(\sigma_{3}\sigma_{1}\sigma_{2}\sigma_{3}^2\underline{\sigma_{1})(\sigma_{2}^{-1}\sigma_{1}^{-1}}\sigma_{1}^{-1}\sigma_{3}^{-2}\sigma_{2})(\sigma_{3}^2\sigma_{4}\sigma_{3}^{-2})x(\sigma_{4}\sigma_{3}^{-1})
\\
&=\hl{\sigma_{2}}(\sigma_{3}\sigma_{4}^{-1})x^{-1}(\sigma_{3}\sigma_{1}\sigma_{2}\underline{\sigma_{3}^2} \sigma_{2}^{-1}\sigma_{1}^{-1}\sigma_{2}\sigma_{1}^{-1}\sigma_{3}^{-2}\sigma_{2})(\sigma_{3}^2 \sigma_{4}\sigma_{3}^{-2})x(\sigma_{4}\sigma_{3}^{-1})
\\
&=\hl{\sigma_{2}}(\sigma_{3}\sigma_{4}^{-1})x^{-1}(\sigma_{3}\sigma_{1}\sigma_{2}\sigma_{3}(\sigma_{2}^{-1}\sigma_{2})\sigma_{3}\sigma_{2}^{-1}\sigma_{1}^{-1}\sigma_{2}\sigma_{1}^{-1}\sigma_{3}^{-2}\sigma_{2})(\sigma_{3}^2 \sigma_{4}\sigma_{3}^{-2})x(\sigma_{4}\sigma_{3}^{-1})
\\
&=\hl{\sigma_{2}}(\sigma_{3}\sigma_{4}^{-1})x^{-1}(\sigma_{3}\sigma_{1}\underline{\sigma_{2}\sigma_{3}\sigma_{2}^{-1}}  \ \underline{\sigma_{2}\sigma_{3}\sigma_{2}^{-1}}\sigma_{1}^{-1}\sigma_{2}\sigma_{1}^{-1}\sigma_{3}^{-2}\sigma_{2})(\sigma_{3}^2 \sigma_{4}\sigma_{3}^{-2})x(\sigma_{4}\sigma_{3}^{-1})
\\
&=\hl{\sigma_{2}}(\sigma_{3}\sigma_{4}^{-1})x^{-1}(\sigma_{3}\sigma_{1}\sigma_{3}^{-1}\sigma_{2}\underline{\sigma_{3}\sigma_{3}^{-1}}\sigma_{2}\sigma_{3}\sigma_{1}^{-1}\sigma_{2}\sigma_{1}^{-1}\sigma_{3}^{-2}\sigma_{2})(\sigma_{3}^2 \sigma_{4}\sigma_{3}^{-2})x(\sigma_{4}\sigma_{3}^{-1})
\\
&=\hl{\sigma_{2}}(\sigma_{3}\sigma_{4}^{-1})x^{-1}(\underline{\sigma_{3}\sigma_{1}\sigma_{3}^{-1}}\sigma_{2}\sigma_{2}\sigma_{3}\sigma_{1}^{-1}\sigma_{2}\sigma_{1}^{-1}\sigma_{3}^{-2}\sigma_{2})(\sigma_{3}^2 \sigma_{4}\sigma_{3}^{-2})x(\sigma_{4}\sigma_{3}^{-1})
\\
&=\hl{\sigma_{2}}(\sigma_{3}\sigma_{4}^{-1})x^{-1}(\underline{\sigma_{3}\sigma_{3}^{-1}} \sigma_{1}\sigma_{2}\sigma_{2}\sigma_{3}\sigma_{1}^{-1}\sigma_{2}\sigma_{1}^{-1}\sigma_{3}^{-2}\sigma_{2})(\sigma_{3}^2 \sigma_{4}\sigma_{3}^{-2})x(\sigma_{4}\sigma_{3}^{-1})
\\
&=\hl{\sigma_{2}}(\sigma_{3}\sigma_{4}^{-1})x^{-1}(\sigma_{1}\sigma_{2}\sigma_{2}\underline{\sigma_{3}\sigma_{1}^{-1}} \sigma_{2}\sigma_{1}^{-1}\sigma_{3}^{-2}\sigma_{2})(\sigma_{3}^2 \sigma_{4}\sigma_{3}^{-2})x(\sigma_{4}\sigma_{3}^{-1})
\\
&=\hl{\sigma_{2}}(\sigma_{3}\sigma_{4}^{-1})x^{-1}(\sigma_{1}\underline{\sigma_{2}\sigma_{2}}\sigma_{1}^{-1}\sigma_{3}\sigma_{2}\sigma_{1}^{-1}\sigma_{3}^{-2}\sigma_{2})(\sigma_{3}^2 \sigma_{4}\sigma_{3}^{-2})x(\sigma_{4}\sigma_{3}^{-1})
\\
&=\hl{\sigma_{2}}(\sigma_{3}\sigma_{4}^{-1})x^{-1}(\sigma_{1}\sigma_{2}(\sigma_{1}^{-1} \sigma_{1})\sigma_{2}\sigma_{1}^{-1} \sigma_{3}\sigma_{2}\sigma_{1}^{-1}\sigma_{3}^{-2}\sigma_{2})(\sigma_{3}^2 \sigma_{4}\sigma_{3}^{-2})x(\sigma_{4}\sigma_{3}^{-1})
\\
&=\hl{\sigma_{2}}(\sigma_{3}\sigma_{4}^{-1})x^{-1}(\underline{\sigma_{1}\sigma_{2}\sigma_{1}^{-1}}  \ \underline{\sigma_{1}\sigma_{2}\sigma_{1}^{-1}} \sigma_{3}\sigma_{2}\sigma_{1}^{-1}\sigma_{3}^{-2}\sigma_{2})(\sigma_{3}^2 \sigma_{4}\sigma_{3}^{-2})x(\sigma_{4}\sigma_{3}^{-1})
\\
&=\hl{\sigma_{2}}(\sigma_{3}\sigma_{4}^{-1})x^{-1}(\sigma_{2}^{-1}\sigma_{1}\underline{\sigma_{2}\sigma_{2}^{-1}}\sigma_{1}\sigma_{2}\sigma_{3}\sigma_{2}\underline{\sigma_{1}^{-1}\sigma_{3}^{-2}}\sigma_{2})(\sigma_{3}^2 \sigma_{4}\sigma_{3}^{-2})x(\sigma_{4}\sigma_{3}^{-1})
\\
&=\hl{\sigma_{2}}(\sigma_{3}\sigma_{4}^{-1})x^{-1}(\sigma_{2}^{-1}\sigma_{1}\sigma_{1}\sigma_{2} \underline{\sigma_{3}\sigma_{2}\sigma_{3}^{-1}}\sigma_{1}^{-1}\sigma_{3}^{-1}\sigma_{2})(\sigma_{3}^2 \sigma_{4}\sigma_{3}^{-2})x(\sigma_{4}\sigma_{3}^{-1})
\\
&=\hl{\sigma_{2}}(\sigma_{3}\sigma_{4}^{-1})x^{-1}(\sigma_{2}^{-1}\sigma_{1}\sigma_{1}\underline{\sigma_{2}\sigma_{2}^{-1}}\sigma_{3}\sigma_{2}\sigma_{1}^{-1}\sigma_{3}^{-1}\sigma_{2})(\sigma_{3}^2 \sigma_{4}\sigma_{3}^{-2})x(\sigma_{4}\sigma_{3}^{-1})
\\
&=\hl{\sigma_{2}}(\sigma_{3}\sigma_{4}^{-1})x^{-1}(\sigma_{2}^{-1}\sigma_{1} \underline{\sigma_{1}\sigma_{3}}\sigma_{2}\sigma_{1}^{-1}\sigma_{3}^{-1}\sigma_{2})(\sigma_{3}^2 \sigma_{4}\sigma_{3}^{-2})x(\sigma_{4}\sigma_{3}^{-1})
\\
&=\hl{\sigma_{2}}(\sigma_{3}\sigma_{4}^{-1})x^{-1}(\sigma_{2}^{-1}\underline{\sigma_{1}\sigma_{3}}\sigma_{1}\sigma_{2}\sigma_{1}^{-1}\sigma_{3}^{-1}\sigma_{2})(\sigma_{3}^2 \sigma_{4}\sigma_{3}^{-2})x(\sigma_{4}\sigma_{3}^{-1})
\\
&=\hl{\sigma_{2}}(\sigma_{3}\sigma_{4}^{-1})x^{-1}(\sigma_{2}^{-1}\sigma_{1}(\sigma_2 \sigma_2^{-1}) \sigma_{3}\sigma_{1}\sigma_{2}\sigma_{1}^{-1}\sigma_{3}^{-1}\sigma_{2})(\sigma_{3}^2 \sigma_{4}\sigma_{3}^{-2})x(\sigma_{4}\sigma_{3}^{-1})
\\
&=\hl{\sigma_{2}}(\sigma_{3}\sigma_{4}^{-1})x^{-1}(\sigma_{2}^{-1}\sigma_{1}\underline{\sigma_{2})(\sigma_{2}^{-1}} \sigma_{3}\sigma_{1}\sigma_{2}\sigma_{1}^{-1}\sigma_{3}^{-1}\underline{\sigma_{2})(\sigma_{3}^2} \sigma_{4}\sigma_{3}^{-2})x(\sigma_{4}\sigma_{3}^{-1})
\\
&=\hl{\sigma_{2}}(\sigma_{3}\sigma_{4}^{-1})x^{-1}(\sigma_{2}^{-1}\sigma_{1}\sigma_{2})(x\sigma_{4}\sigma_{3}^{-1}\sigma_{3}\sigma_{4}^{-1}x^{-1})(\sigma_{2}^{-1}\sigma_{3}\sigma_{1}\sigma_{2}\sigma_{1}^{-1}\sigma_{3}^{-1}\sigma_{2})\\
& \qquad \qquad (x\sigma_{4}\sigma_{3}^{-1}\sigma_{3}\sigma_{4}^{-1}x^{-1}) (\sigma_{3}^2 \sigma_{4}\sigma_{3}^{-2})x(\sigma_{4}\sigma_{3}^{-1})
\\
&=\hl{\sigma_{2}}(\underline{\sigma_{3}\sigma_{4}^{-1}x^{-1}}\sigma_{2}^{-1}\sigma_{1}\sigma_{2}\underline{x\sigma_{4}\sigma_{3}^{-1}})(\underline{\sigma_{3}\sigma_{4}^{-1}x^{-1}}\sigma_{2}^{-1}\sigma_{3}\sigma_{1}\sigma_{2}\sigma_{1}^{-1}\sigma_{3}^{-1}\sigma_{2} \underline{x\sigma_{4}\sigma_{3}^{-1}}) \\
& \qquad \qquad (\underline{\sigma_{3}\sigma_{4}^{-1}x^{-1}}\sigma_{3}^2 \sigma_{4}\sigma_{3}^{-2}\underline{x\sigma_{4}\sigma_{3}^{-1}})
\end{align*}

Finally, plugging in $w=\sigma_3 \sigma_4^{-1} x^{-1}=\sigma_3\sigma_4^{-1} \sigma_1^{-1} \sigma_3^{-2} \sigma_2^{-1} \sigma_1^{-1} \sigma_3^{-1}$ yields
\begin{align*}
\quad \cdots \ &=\hl{\sigma_{2}}(w \sigma_2^{-1} \sigma_{1} \sigma_2 w^{-1})(w \sigma_2^{-1} \sigma_{3}\sigma_{1}\sigma_{2}\sigma_{1}^{-1}\sigma_{3}^{-1} \sigma_2 w^{-1}) ( w \sigma_{3}^2 \sigma_{4}\sigma_{3}^{-2} w^{-1}) = \beta' \hspace{3.75em}  \qedhere
\end{align*}
\end{proof}

\subsection{Additional handle diagrams.}\label{sec:handles}

Figures~\ref{fig:exterior}-\ref{fig:torus-dbc} below produce handle diagrams for the double branched covers of $B^4$ along the annulus $A$ and torus $T$ for Lemma~\ref{lem:DBCs}.

\smallskip

\begin{figure}[h!]\center
\def\svgwidth{.975\linewidth}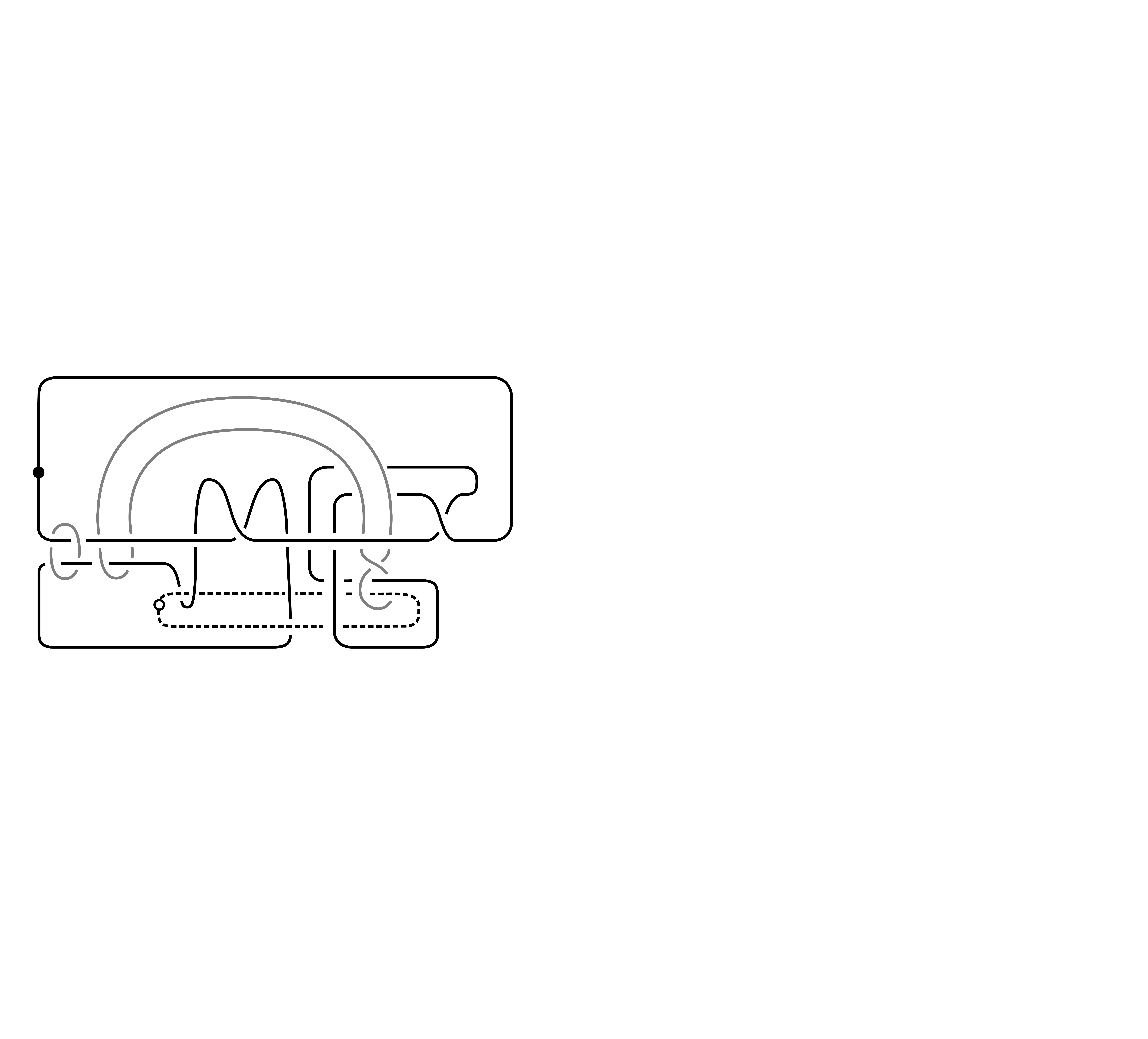
\caption{Simplifying a handle diagram for the exterior of the braided surface $A \subset B^4$.}\label{fig:exterior}
\end{figure}

\newgeometry{margin=1.1in, left=1.25in, right=1.15in}
\titleformat{\section}  {\normalfont\fontsize{12}{5}\bfseries}{}{\center 1em}{\center}

\begin{figure}[h!]\center
\smallskip
\def\svgwidth{.9\linewidth}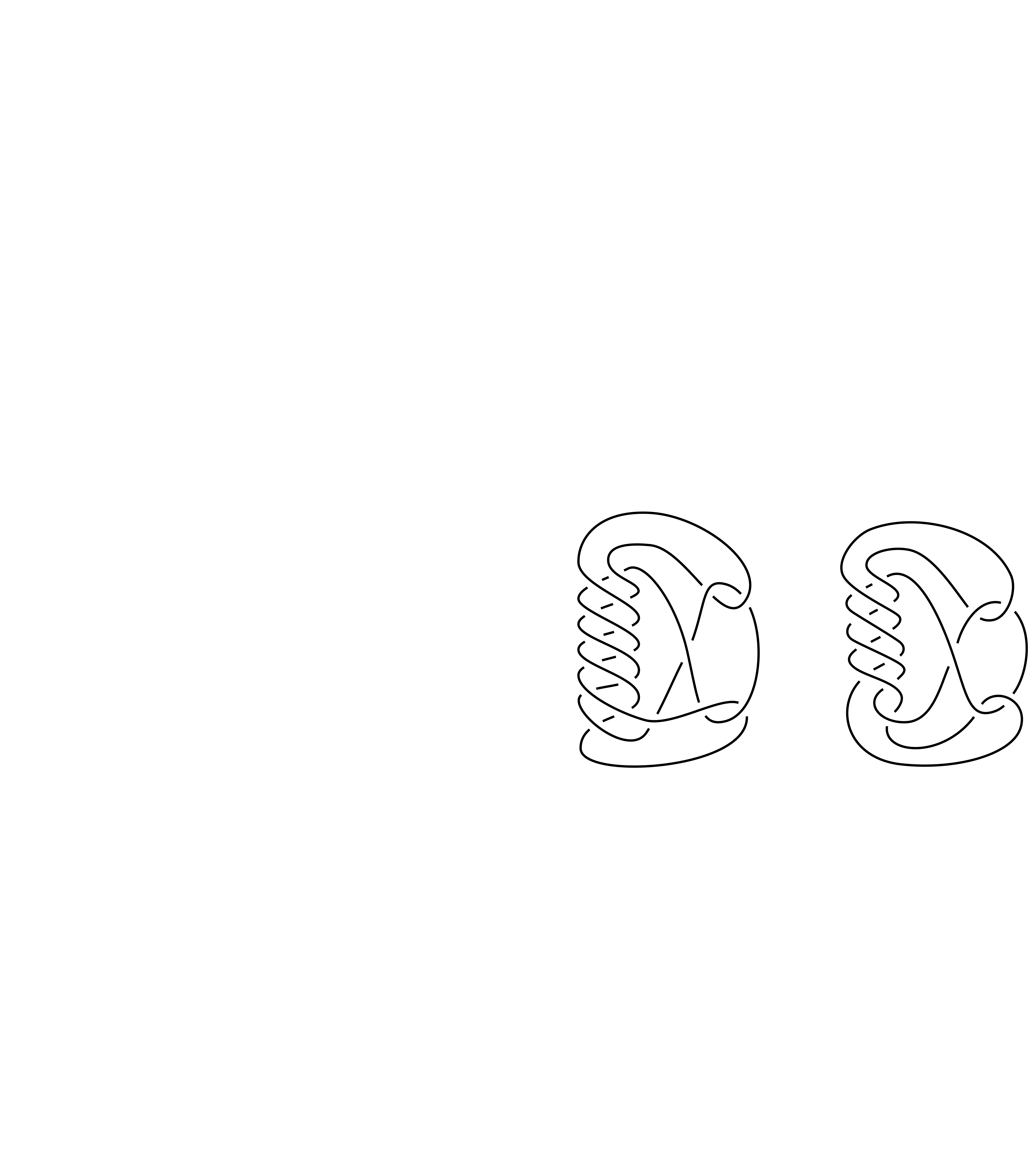
\caption{Simplifying the handle diagram for the branched double cover of $(B^4,A)$.}\label{fig:dbc}
\end{figure}

\begin{figure}\center
\smallskip
\def\svgwidth{.85\linewidth}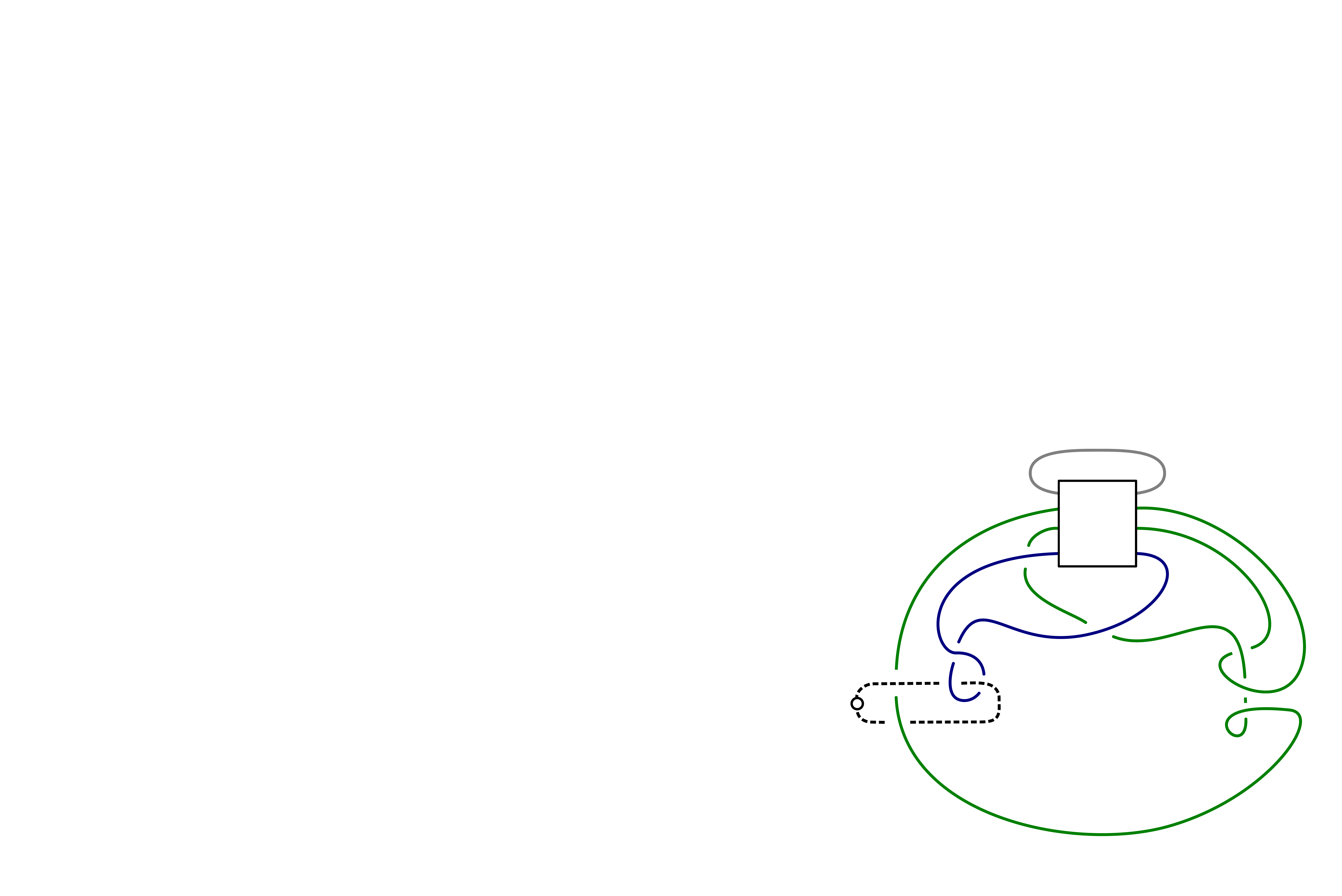
\caption{Parts (a) and (b) show handle diagrams for the exterior of the torus $T$ in $B^4$. Parts (c) and (d) depict steps in the simplification of the handle diagram for $\Sigma(B^4,T)$.}\label{fig:torus-dbc}
\end{figure}

\vspace{-.1in}

{\small \footnotesize \bibliographystyle{alphamod}
\bibliography{biblio}}

\end{document}